\documentclass{article}
\usepackage{xcolor}
\usepackage[utf8x]{inputenc}
\usepackage{amssymb,latexsym}
\usepackage{amsmath}
\allowdisplaybreaks
\usepackage{graphicx}
\usepackage{algorithm,algorithmicx}
\usepackage[margin=2cm]{geometry}
\usepackage[noEnd=false]{algpseudocodex}
\usepackage{amsthm}
\usepackage{breqn}
\usepackage{comment}
\usepackage{diagbox} 
\usepackage{array} 
\usepackage{amssymb}
\usepackage{pifont}
\usepackage{fontawesome}
\usepackage{cancel}
\usepackage{forest}
\usepackage{amsthm,amssymb,amsmath,enumerate,graphicx, tikz}
\usepackage{amscd}
\usepackage{setspace}
\usepackage{comment}
\usepackage{hyperref}
\usepackage{ytableau}
\usepackage{mathtools}

\renewcommand{\geq}{\geqslant}
\renewcommand{\leq}{\leqslant}
\renewcommand{\ge}{\geqslant}
\renewcommand{\le}{\leqslant}

\newcommand{\SC}{\textnormal{sc}}
\newcommand{\SR}{\textnormal{sr}}
\newcommand{\SE}{\textnormal{se}}
\newcommand{\I}{\mathcal{I}}

\newtheorem{theorem}{Theorem}

\newtheorem{conjecture}[theorem]{Conjecture}
\newtheorem{corollary}[theorem]{Corollary}

\newtheorem{claim}{Claim}

\newtheorem{observation}[claim]{Observation}
\newtheorem{notation}[claim]{Notation}
\newtheorem{definition}[theorem]{Definition}

\theoremstyle{remark}
\newtheorem{remark}[theorem]{Remark}

\def\eref#1{$(\ref{#1})$}
\def\sref#1{\S$\ref{#1}$}

\def\tref#1{Theorem~$\ref{#1}$}
\def\fref#1{Figure~$\ref{#1}$}
\def\cref#1{Conjecture~$\ref{#1}$}
\def\dref#1{Definition~$\ref{#1}$}
\def\clref#1{Claim~$\ref{#1}$}

\title{Outline Rectangles, Allocations, and Latin Young Diagrams}
\author{Jack Allsop, Dani Kotlar and Ian M. Wanless}
\date{}

\begin{document}

\maketitle

\begin{abstract}
A Young diagram is \emph{Latin} if there is an assignment of integers to its cells so that each row $i$ of length $l_i$ is populated by the numbers $1,\ldots,l_i$, and the numbers in each column are distinct. 
A Young diagram is called \emph{wide} if any subdiagram, formed by a subset of its rows, dominates its conjugate.
Chow et al. [Advances in Applied Mathematics, 31, 2003] conjectured that any wide Young diagram is Latin. 
We introduce a notion of an \emph{allocation} which can be thought of
as a coarse attempt at finding a Latin filling for a Young diagram.
Using a theorem of Hilton we prove that a Young diagram has an
allocation if and only if it is Latin.  This enables us to prove Chow et al.'s
conjecture for Young diagrams with three distinct row lengths.  
\end{abstract}

\section{Introduction}\label{sec:introduction}

Young diagrams are a way of representing partitions of natural numbers. Given a partition  $n=l_1+l_2+\cdots + l_m$ of a number $n$, where $0\le l_1 \le l_2 \le \cdots \le l_m$, the corresponding Young diagram~$Y$ is formed by $m$ rows of squares that are left-aligned, where the $i$-th row from the bottom consists of $l_i$ squares. 
Chow, Fan, Goemans, and Vondrak posed the following question \cite{chow2003}:

Suppose the squares of a Young diagram are filled, each with an element of the ground set of a given matroid $M$, so that the rows are independent sets in $M$. Can the elements in each row be rearranged so that the columns are independent? 

Their conjecture, named the Wide Partition Conjecture, is that the answer is positive if and only if the Young diagram is so-called \emph{wide}, where wideness is defined as follows:

\begin{definition}\label{def:wide}
    A Young diagram is wide if any subdiagram formed by a subset of its rows dominates its conjugate.
\end{definition}

The `only if' part of their conjecture can easily be shown to hold, so the conjecture is really about the `if' part.
Handling this problem for general matroids is hard, since even the `easy' case in which all rows of the given Young diagram have the same length is a notorious open problem known as Rota's basis conjecture \cite{rota94}. So, as in \cite{chow2003}, we focus on the simple case of a free matroid.
A Young diagram is \emph{Latin} if there is an assignment of integers to its cells so that each row $i$ of length $l_i$ is populated by the numbers $1,\ldots,l_i$, and the assignment is injective in each column. We call such an assignment a \emph{Latin filling}.
Chow et al. showed that in the case of a free matroid, the Wide Partition Conjecture can be reduced to the following conjecture:

\begin{conjecture}\label{conj:wpc}
    Every wide Young diagram is Latin. 
\end{conjecture}

In \cite{chow2003}, Conjecture~\ref{conj:wpc} was proved for Young diagrams with two distinct row lengths and for Young diagrams with three distinct row lengths that are self-conjugate. It was also shown there that proving Conjecture~\ref{conj:wpc} for all self-conjugate Young diagrams will imply the general case. However, these results do not imply that Conjecture~\ref{conj:wpc} holds for all Young diagrams with three distinct row lengths.
One of the main results of this work is a proof of Conjecture~\ref{conj:wpc} for such Young diagrams:

\begin{theorem}\label{thm:three}
    A wide Young diagram with three distinct row lengths is Latin.
\end{theorem}

To prove \tref{thm:three}, we introduce the notion of an \emph{allocation}, which is a coarse filling of a Young diagram adhering to certain constraints. Instead of specifying which symbol goes in each cell, we subdivide the Young diagram into sub-rectangles and partition the symbol set into smaller subsets. An allocation specifies how many elements of each symbol subset should be in each sub-rectangle. We defer the formal definition of an allocation to \sref{sec:outline}. Our approach is inspired by the notion of an \emph{outline rectangle}, introduced by Hilton \cite{Hilton} for rectangular arrays. We rely on Hilton's result, which states that any outline rectangle can be traced back to a Latin square, to demonstrate that if a Young diagram has an allocation, then it has a Latin filling. Thus, another of our main results is that we show that to prove Conjecture~\ref{conj:wpc} it suffices to prove the following equivalent conjecture:

\begin{conjecture}\label{conj:wide_implies_outline}
    Any wide Young diagram has an allocation.
\end{conjecture}
We prove Conjecture~\ref{conj:wide_implies_outline} for Young diagrams with three distinct row lengths. 
The method of proof suggests that it might be possible to prove Conjecture~\ref{conj:wide_implies_outline} in the general case by induction on the number of distinct row-lengths. 

Two recent works applied other approaches to the problem of Latin fillings of Young diagrams. Chow and Tiefenbruck~\cite{chow_tief} made some progress on the so-called Latin Tableau conjecture of Chow et al.~\cite{chow2003}, which is a generalisation of Conjecture~\ref{conj:wpc}.  
Aharoni et al.~\cite{abgk} took a hypergraph approach and showed that Conjecture~\ref{conj:wpc} would follow if equality holds between a matching parameter of a 3-partite 3-hypergraph related to a given Young diagram, and its covering dual. 

This paper is arranged as follows: In Section~\ref{sec:wideness}, we show that the complexity of determining wideness of a Young diagram is at most quadratic in the number of distinct row lengths.
In Section~\ref{sec:outline}, we introduce allocations and use a theorem of Hilton to show that Conjecture~\ref{conj:wpc} is equivalent to Conjecture~\ref{conj:wide_implies_outline}.
In Section~\ref{sec:three}, we prove that any wide Young diagram with three distinct row lengths has an allocation.
In Section~\ref{sec:general}, we discuss two approaches to proving Conjecture~\ref{conj:wide_implies_outline} in general. 

\section{Testing wideness}\label{sec:wideness}

Throughout this section, let $Y$ be a Young diagram with $p$ distinct row lengths $a_1<a_2<\cdots<a_p$ and for $i=1,\ldots,p$, let the number of rows of $Y$ of length $a_i$ be $e_i$. Also set $a_0=0$. 

\begin{definition}
    A row block in $Y$ is a maximal set of rows of the same length.
    A column block is a maximal set of columns of the same length.
    The symbol blocks are  $S_1, S_2, \ldots , S_p$, where $S_i$ is the set of symbols $\{a_{i-1}+1, a_{i-1}+2, \ldots, a_i\}$.
    When row/column/symbol is not specified, we shall assume that `block' refers to a row block.
\end{definition}

The row blocks of $Y$ are numbered from bottom to top (from short to long), while the column blocks are numbered from left to right (from long to short). As defined, the symbol blocks are numbered according to the ascending order of symbols. 

Considering Definition~\ref{def:wide}, which subdiagrams do we need to check to determine whether $Y$ is wide? Chow et al.~\cite{chow2003} showed that it suffices to consider `tail' subdiagrams (i.e.,~subdiagrams consisting of all rows from some row $r$ downwards), and check whether the top $w$ rows of each such subdiagram dominates the leftmost $w$ columns. However, it seems we do not need to consider all values of $r$ and $w$. 

\begin{notation}
    Denote by $Y[r]$ the tail subdiagram of $Y$ consisting of all the rows from $r$ downwards.
\end{notation}

Throughout this paper, when $a$ and $b$ are integers with $a \leq b$ we will use the interval notation $[a, b] = \{a, a+1, \ldots, b\}$. We will use $[a]$ as shorthand for $[1, a]$.

\begin{theorem}\label{thm:widecheck:1}
To check whether $Y$ is wide it suffices to check whether the top $w$ rows of the subdiagram $Y[r]$ dominates its leftmost $w$ columns, whenever $r$ is the top row of a block $k$ and $w\in\{a_1,\ldots, a_{k}\}$.
\end{theorem}

\begin{proof}
First fix $r$ in some block $k$ and consider varying $w$ within some interval $[a_i,a_{i+1}]$, where $1\le i<k$ (there are only $a_k$ columns in $Y[k]$ so we do not need to consider $w>a_k$),
or within the interval $[1,a_1]$. Let $D$ be the sum of the lengths of the top $w$ rows of $Y[r]$ minus the sum of the lengths of the leftmost $w$ columns of $Y[r]$. Wideness dictates that $D$ must be non-negative. As we increase $w$ by $1$, we add to $D$ the length of one new row and subtract the length of one new column. If this reduces $D$ then each subsequent increase in $w$ will also decrease $D$ because the new rows get shorter but the new columns have the same length (since we are within the interval $[a_i,a_{i+1}]$). So $D$ will be minimized by taking $w=a_{i+1}$, which is the only value of $w$ that we need to test.

On the other hand, if $D$ increases when we add 1 to $w$, then $D$ will decrease each time we reduce $w$ (the rows we are discarding are getting bigger, and the columns stay the same), so $D$ will be minimized by taking $w=a_i$ or, in the case where $w \in [a_1]$, $D$ will be minimized by taking $w=1$.
Either way, we only need to test $w$ at one end of the interval.

Next, we consider fixing $w$ and moving $r$ within the interval of $e_k$ rows of the same length as $r$.
The argument is similar.
As we move $r$ down the diagram, we add a row of some length $\ell$, subtract a row of length $a_k$ and subtract $w$ from the column total (meaning $D$ increases by $\ell-a_k+w$). 
If this reduces $D$, then each subsequent time we move $r$ downwards will also decrease $D$, because $\ell$ decreases and $w-a_k$ is constant. So we minimize $D$ by moving $r$ to the top row of the $(k-1)$-st block. 
On the other hand, if $\ell-a_k+w\ge0$, we should move $r$ to the top of the $k$-th block. 

We claim that for a given $r$, the top row of row block $k$, we do not need to check whether the top $w$ rows of $Y[r]$ dominates its leftmost $w$ columns for both $w=1$ and $w=a_k$, as the domination conditions with $w=1$ and $w=a_k$ are equivalent.
First, note that for $w=1$, the condition is $\sum_{i=1}^ke_i\le a_k$, which means that in $Y[r]$, there are at most $a_k$ rows. 
Also note that there are $a_k$ columns in $Y[r]$, so the sum of the lengths of the leftmost $a_k$ columns of $Y[r]$ is $|Y[r]|$.
Suppose $Y[r]$ passes the domination test with $w=1$. 
That is, $\sum_{i=1}^ke_i\le a_k$. Then, $Y[r]$ has at most $a_k$ rows. So, the sum of the lengths of the top $a_k$ rows and the sum of the lengths of the leftmost $a_k$ columns of $Y[r]$ are both equal to $|Y[r]|$. 
Thus, $Y[r]$ passes the domination test with $w=a_k$. 
Now, assume $Y[r]$ does not satisfy the domination condition with $w=1$. 
That is, $\sum_{i=1}^ke_i> a_k$. 
This means that there are more than $a_k$ rows in $Y[r]$. 
So, the sum of the lengths of the top $a_k$ rows in $Y[r]$ is less than $|Y[r]|$. 
Thus, $Y[r]$ does not satisfy the domination condition with $w=a_k$.
\end{proof}

The next observation follows from the last part of the proof of Theorem~\ref{thm:widecheck:1}.
\begin{observation}\label{obs:condition_for_ak}
    If $r$ is the top row of a block $k$, the wideness condition for $Y[r]$ with $w=a_k$ is equivalent to $\sum_{i=1}^ke_i\le a_k$.
\end{observation}

For a given row $r$, the top row in block $k$, and $w=a_j$ where $1\le j\le k-1$, what is the condition to be verified if we want to establish wideness? 

\begin{notation}
    For $k\le p$, denote by $Y_k$ the subdiagram $Y[r]$ where $r$ is the top row in block $k$. 
    For $j\le k$ we denote by $\SC_{j,k}$ (resp.\
    $\SR_{j,k}$) the sum of the lengths of the leftmost $a_j$ columns (resp.\ top $a_j$ rows) in $Y_k$. 
\end{notation}

\begin{claim}
    We have: 
    \begin{equation}\label{eq:sc:1}
        \SC_{j,k}=\sum_{t=1}^{j-1}a_te_t+a_j\sum_{t=j}^ke_t
    \end{equation}
    and also
\begin{equation}\label{eq:sc:2}
        \SC_{j,k}=\sum_{u=1}^j\left((a_u-a_{u-1})\sum_{v=u}^{\smash{k}}e_v\right).
    \end{equation}   
\end{claim}
\begin{proof}
    The first identity is obtained by partitioning the area covered by the leftmost $a_j$ columns of $Y_k$ into horizontal blocks of length at most $a_j$.
    The second identity is obtained by partitioning the same area into vertical blocks of widths $a_1, a_2-a_1,\ldots, a_j-a_{j-1}$.
\end{proof}

The value of $\SR_{j,k}$ depends on the relation between 
$a_j$ and the $e_i$'s. 

\begin{notation}
    For a diagram with $k$ row blocks and $m\in[1,k]$ let $\SE_{m,k}:=\sum_{t=m}^ke_t$, the number of rows in the upper $k-m+1$ row blocks.
\end{notation}

\begin{claim}
    We have
    \begin{equation}\label{eq:sr:1}
        \SR_{j,k}=
        \begin{cases}
            a_ja_k & \text{if }a_j\le e_k,\\
            \sum_{t=i}^ke_ta_t+(a_j-\SE_{i,k})a_{i-1} & \text{if }\SE_{i,k}< a_j\le \SE_{i-1,k}\text{ for some }i\in[2,k],\\ 
            |Y_k| & \text{if } \SE_{1,k}\le a_j.
        \end{cases}    
    \end{equation}    
\end{claim}
\begin{proof}
    The first and third cases are obvious. For the second case, note that since $a_j>\SE_{i,k}$, the subdiagram consisting of the top $a_j$ rows contains the blocks $i,\ldots,k$ entirely, and  $a_j - \SE_{i,k}$ rows in block $i-1$.
\end{proof}

\begin{theorem}\label{thm:wideness_check}
    To check whether $Y$ is wide, we need to perform no more than the following tests:
\begin{equation}\label{eq:gen_wideness:1}
       a_k \ge \sum_{i=1}^ke_i,\;\;\;\text{ for }k=1,\ldots,p,
\end{equation}
    and for $k=1,\ldots,p$, and $j=1,\ldots,k-1$,
\begin{equation}\label{eq:gen_wideness:2}
         \sum_{t=i}^ke_ta_t+(a_j-\SE_{i,k})a_{i-1}\ge \sum_{t=1}^{j-1}a_te_t+a_j\sum_{t=j}^ke_t\quad\text{if}\quad \SE_{i,k}< a_j\le \SE_{i-1,k}\quad\text{for}\quad i\in[2,k]. 
    \end{equation}
\end{theorem}
 \begin{proof}
     By Theorem~\ref{thm:widecheck:1}, to establish wideness, we need to show that $\SR_{j,k}\ge \SC_{j,k}$ for all $1\le j\le k\le p$.
     By Observation~\ref{obs:condition_for_ak}, the inequality \eqref{eq:gen_wideness:1} is the case where $w=a_k$. 
     When $w\in\{a_1,\ldots, a_{k-1}\}$ and $\SE_{i,k}\le a_j\le \SE_{i-1,k}$ for some $i \in [2, k]$, then applying \eqref{eq:sc:1} and \eqref{eq:sr:1} to $\SR_{j,k}\ge \SC_{j,k}$ yields \eqref{eq:gen_wideness:2}. Now suppose that $a_j \leq e_k$. We prove that $\SR_{j,k}\ge \SC_{j,k}$ is implied by \eqref{eq:gen_wideness:1}. 
     To see this, note that by \eqref{eq:sc:2} when $a_j\le e_k$ the inequality $\SR_{j,k}\ge \SC_{j,k}$ is equivalent to 
     \begin{equation}\label{eq:gen_wideness:2:1}
         a_ja_k\ge \sum_{u=1}^j\left((a_u-a_{u-1})\sum_{v=u}^ke_v\right).
     \end{equation}
      But $$\sum_{u=1}^j\left((a_u-a_{u-1})\sum_{v=u}^ke_v\right)\le \sum_{u=1}^j\left((a_u-a_{u-1})\sum_{v=1}^ke_v\right)=a_j\sum_{v=1}^ke_v.$$ So, if \eqref{eq:gen_wideness:1} holds, then $a_ja_k\ge a_j\sum_{v=1}^ke_v$ holds, and this is stronger than \eqref{eq:gen_wideness:2:1}.
      If $\SE_{1,k}\le a_j$, then $\SR_{j,k}\ge \SC_{j,k}$ gives a trivial domination inequality. 
 \end{proof}   

\begin{corollary}\label{cor:wideness_check}
    The number of domination checks required to test whether a Young diagram is wide is at most a quadratic function of the number of distinct row lengths.
\end{corollary}

\section{Outline rectangles and allocations}\label{sec:outline}

In \cite{Hilton} Hilton introduced the notion of an \emph{outline rectangle}, related to Latin squares:

\begin{definition}\label{d:OR}
	Let $C$ be an $m \times m$ matrix of multisets of symbols from $\{\tau_1, \tau_2, \ldots, \tau_m\}$. For $i \in [m]$ let $\rho_i$ be the sum of the cardinalities of the multisets in row $i$, let $c_i$ be the sum of the cardinalities of the multisets in column $i$, and let $\sigma_i$ be the number of occurrences of symbol $\tau_i$ among all the multisets. Then $C$ is an outline rectangle if there is some integer $n$ such that the following is satisfied:
	\begin{enumerate}[(i)]
		\item $n$ divides each of $\rho_i$, $c_i$, and $\sigma_i$,
		\item cell $(i, j)$ contains $\rho_ic_j/n^2$ symbols,
		\item the number of occurrences of $\tau_k$ in row $i$ is $\rho_i\sigma_k/n^2$, and
		\item the number of occurrences of $\tau_k$ in column $j$ is $c_j\sigma_k/n^2$.
	\end{enumerate}
\end{definition}
(The definition in \cite{Hilton} is more general, as the numbers of rows, columns, and symbols in $C$ may not be equal, but here we shall only need this special version).

Hilton showed in \cite{Hilton} that an outline rectangle can be obtained from any $n\times n$ Latin square $L$ as follows. Let $P=(p_i)_{i=1}^m$, $Q=(q_i)_{i=1}^m$, and $S=(s_i)_{i=1}^m$ be sequences of positive integers satisfying $\sum_{i=1}^mp_i=\sum_{i=1}^mq_i=\sum_{i=1}^ms_i=n$. Construct the outline rectangle $C$ from $L$ by amalgamating rows  
$p_0+\cdots +p_{i-1}+1,\ldots,p_0+\cdots +p_i$, columns $q_0+\cdots +q_{i-1}+1,\ldots,q_0+\cdots +q_i$, and symbols $s_0+\cdots +s_{i-1}+1,\ldots,s_0+\cdots +s_i$, for $i=1,\ldots,m$, where we interpret $p_0=q_0=s_0=0$. 
The array $C$ thus obtained is called \emph{the reduction modulo $(P,Q,S)$ of $L$}. 
Hilton~\cite{Hilton} also showed that this works both ways, that is, we can obtain a Latin square from any outline rectangle:

\begin{theorem}\label{thm:Hilton}
    Each outline rectangle $C$ is the reduction modulo $(P,Q,S)$ of some Latin square $L$, for some sequences $P$, $Q$, and $S$.
\end{theorem}

Inspired by this idea, we introduce the notion of an \emph{allocation}:

    \begin{figure}[h!]
    \centering
    \includegraphics[width=6.0cm]{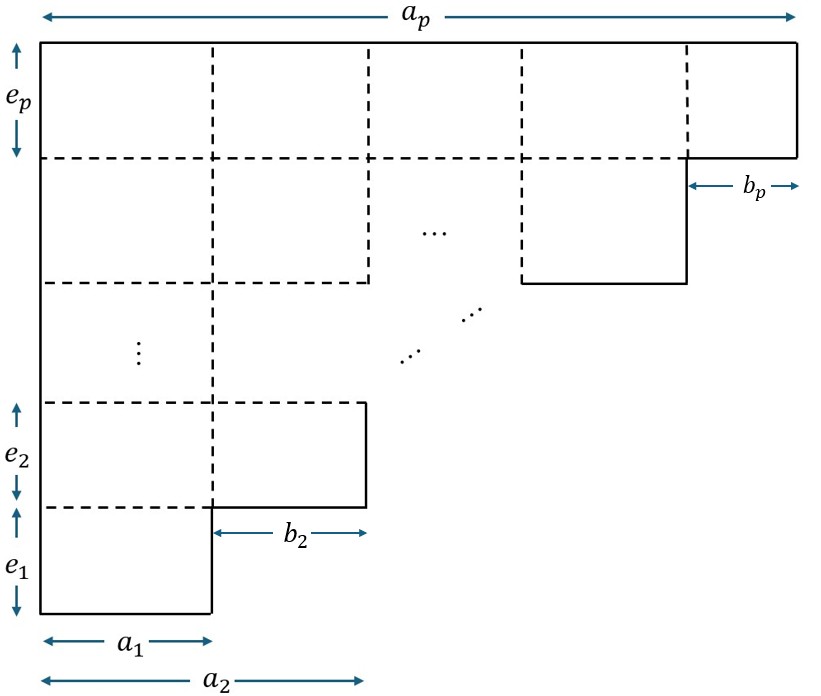}
    \caption{Rectangular regions for an allocation.}
    \label{fig:allocation}
    \end{figure}

\begin{definition}\label{d:allocation}
Let $Y$ be a Young diagram consisting of $e_i$ rows of size $a_i$ for $1\le i\le p$, where $a_1 < a_2 < \cdots < a_p$, as shown in \fref{fig:allocation}. Define $a_0=0$ and $b_i=a_i-a_{i-1}$ for $1\le i\le p$. An \emph{allocation} for $Y$ is a collection of nonnegative integers $z_{ijk}$ for $1\le j,k\le i\le p$, satisfying
\begin{align}
&\sum_{k=1}^i z_{ijk}=e_ib_j&&\text{for }1\le j\le i\le p,\label{e:sumk}\\
&\sum_{j=1}^i z_{ijk}=e_ib_k&&\text{for }1\le k\le i\le p,\label{e:sumj}\\
&\smashoperator[l]{\sum_{i=\max(j,k)}^p} z_{ijk}\le b_jb_k&&\text{for }1\le j,k\le p.\label{e:sumi}
\end{align}

\end{definition}

Our first major result regarding allocations states that wideness is a necessary condition for the existence of an allocation.

\begin{theorem}\label{t:allocwide}
    If a Young diagram $Y$ has an allocation, then it is wide. 
\end{theorem}

\begin{proof}    
    Assuming that $Y$ has an allocation, we will use \tref{thm:widecheck:1} to
establish that $Y$ is wide. Let $1\le j_0\le
i_0\le p$ and consider the subdiagram $Y' = Y_{i_0}$. The
total number of cells in the leftmost $a_{j_0}$ columns of $Y'$ is
\begin{equation}\label{e:zijkact}
  \sum_{j=1}^{j_0}\sum_{i=j}^{i_0}b_je_i
  =\sum_{k=1}^{j_0}\sum_{i=k}^{i_0}b_ke_i
  =\sum_{k=1}^{j_0}\sum_{i=k}^{i_0}\sum_{j=1}^{i}z_{ijk}
  =\sum_{j=1}^{i_0}\sum_{k=1}^{j_0}\sum_{i=\max(j,k)}^{i_0}z_{ijk},
\end{equation}
by \eref{e:sumj}. Now, by \eref{e:sumi},
\begin{equation}\label{e:zijkapp1}
  \sum_{k=1}^{j_0}\sum_{i=\max(j,k)}^{i_0}z_{ijk}
  \le\sum_{k=1}^{j_0}b_jb_k=b_j\sum_{k=1}^{j_0}b_k=b_ja_{j_0}.
\end{equation}
On the other hand, 
\begin{equation}\label{e:zijkapp2}
  \sum_{k=1}^{j_0}\sum_{i=\max(j,k)}^{i_0}z_{ijk}
  \le\sum_{k=1}^{i_0}\sum_{i=\max(j,k)}^{i_0}z_{ijk}
  =\sum_{i=j}^{i_0}\sum_{k=1}^{i}z_{ijk}
  =\sum_{i=j}^{i_0}b_je_i,
\end{equation}
by \eref{e:sumk}. Combining \eref{e:zijkapp1} and \eref{e:zijkapp2} we find
that \eref{e:zijkact} is no more than
\begin{equation*}
\sum_{j=1}^{i_0}b_j\min\left(a_{j_0},\sum_{i=\smash{j}}^{i_0}e_i\right),
\end{equation*}
which is the number of cells in the top $a_{j_0}$ rows of $Y'$.  
    It follows that $Y$ is wide. 
\end{proof}

If \cref{conj:wide_implies_outline} is true, then the converse of \tref{t:allocwide} also holds.

Our next major result on allocations is an analogue of Theorem~\ref{thm:Hilton} for Young diagrams.

\begin{theorem}\label{thm:alloc_implies_latin}
A Young diagram $Y$ has an allocation if and only if $Y$ is Latin.
\end{theorem}

\begin{proof}
If $Y$ is Latin, then it has an allocation; specifically, given a Latin filling of $Y$ let $z_{ijk}$, for $1\le j,k\le i\le p$, be the number of symbols from symbol block $k$ that appear in the rectangular region in $Y$ where row block $i$ and column block $j$ intersect (as illustrated in Figure~\ref{fig:allocation}).

    Suppose that $Y$ consists of $e_i$ rows of size $a_i$ for $1\le i\le p$, where $a_1 < a_2 < \cdots < a_p$. Henceforth, we may assume that $Y$ has an allocation $Z'=(z'_{ijk})_{1\le j,k\le i\le p}$. We shall embed the underlying array of $Z'$ in a $(p+1)\times(p+1)$ array and complete this array to an outline rectangle, as illustrated in Figure~\ref{fig:embedding}. 

    \begin{figure}[h!]
    \centering
    \includegraphics[width=4.0cm]{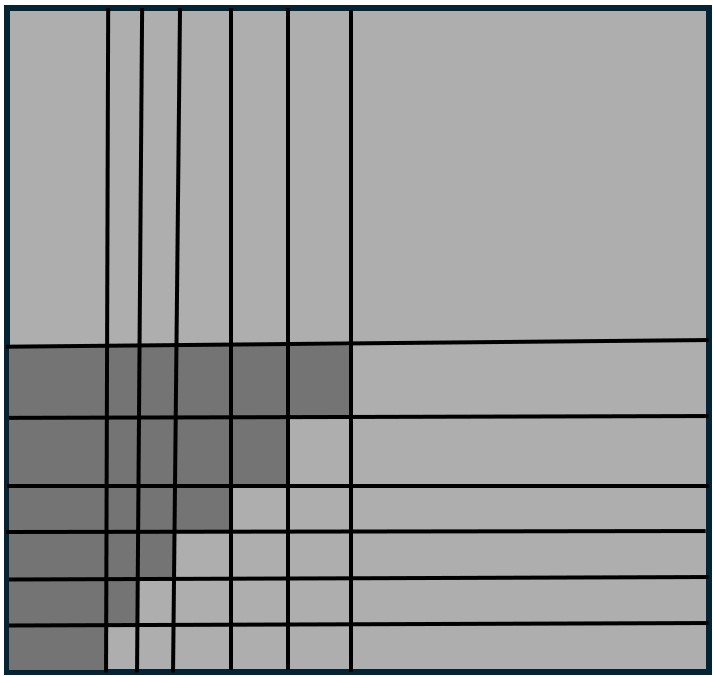}
    \caption{Embedding an allocation in an outline rectangle. The size of each cell is indicative of the size of the multiset that is assigned to it.}
    \label{fig:embedding}
    \end{figure}
        
    Let $n=2a_p$, let $b_{p+1} = a_p$, and let $e_{p+1} = n-(e_1+e_2+\cdots+e_p)$.
    Define
\begin{equation}\label{eq:def:zijk}
    \begin{split}
        z_{ijk} = \begin{cases}
        z'_{ijk} & \text{if } \max(j,k) \leq i \leq p,\\        
        0 & \text{if } i <\max(j,k) \leq p, \text{ or } j \leq i \leq p<k=p+1, \\        
        e_ib_j & \text{if } i < j \leq p< k=p+1,\\        
        e_ib_k-(z_{i1k}+z_{i2k}+\cdots+z_{ipk}) & \text{if } i < j=p+1, \\        
        b_jb_k-(z_{1jk}+z_{2jk}+\cdots+z_{pjk}) & \text{if } i=p+1.		 
        \end{cases}
    \end{split}    
\end{equation}

    We verify that each $z_{ijk}$ is nonnegative. 
    The first three cases in \eqref{eq:def:zijk} are obvious, since $(z'_{ijk})$ is an allocation. 
    For the fourth case, suppose $k\le i \leq p$. Then
    \[
    e_ib_k-\sum_{\ell=1}^{p}z_{i\ell k} = e_ib_k-\sum_{\ell=1}^{i}z'_{i\ell k} = e_ib_k-e_ib_k=0,
    \]
    from \eref{e:sumj}. If $p\ge k > i$ then
    \[
    e_ib_k-\sum_{\ell=1}^{p}z_{i\ell k} =e_ib_k \geq 0.
    \]
    Finally, 
    \[
    e_ib_{p+1}-\sum_{\ell=1}^{p}z_{i\ell(p+1)} = e_ib_{p+1}-\sum_{\ell=i+1}^{p} e_ib_\ell = e_i\left(b_{p+1}-\sum_{\smash{\ell}=i+1}^{p}b_\ell\right) \geq 0.
    \]
    The last inequality holds since  $b_{p+1}=a_p$.
    We conclude that $z_{i(p+1)k} \geq 0$ whenever $i\le p$. 
    
    For the fifth case, let $i=p+1$. If $\max(j,k)\le p$ then, from \eref{e:sumi}
    \[
    b_jb_k-\sum_{\ell=1}^{p} z_{\ell jk} = b_jb_k-\sum_{\ell=\max(j,k)}^{p} z'_{\ell jk} \geq b_jb_k-b_jb_k = 0.
    \]
    If $k = p+1$, then
    \[
    b_jb_{k}-\sum_{\ell=1}^{p} z_{\ell jk} \ge b_jb_{p+1}-\sum_{\ell=1}^{p}e_\ell b_j = b_j\left(b_{p+1}-\sum_{\ell=1}^{p}e_\ell\right) \geq 0,
    \]
    since $b_{p+1}=a_p$ and $Y$ is wide. Similarly, if $j=p+1$, then
    \[
    b_jb_{k}-\sum_{\ell=1}^{p} z_{\ell jk} \ge b_{p+1}b_k-\sum_{\ell=1}^{p}e_\ell b_k = b_k\left(b_{p+1}-\sum_{\ell=1}^{p}e_\ell\right) \geq 0.
    \]
    We have shown that $z_{ijk} \geq 0$ for all $i$, $j$ and $k$.
    
    Let $\{\tau_1, \tau_2, \ldots, \tau_{p+1}\}$ be a set of $p+1$ distinct symbols and let $C$ be the $(p+1) \times (p+1)$ matrix where cell $(i, j)$ contains $z_{ijk}$ occurrences of symbol $\tau_k$ for each $k \in [p+1]$. 
    We will prove that $C$ is an outline rectangle. For $i \in [p+1]$, let $\rho_i$ be the number of symbols that occur in row $i$, let $c_i$ be the number of symbols that occur in column $i$, and let $\sigma_i$ be the number of occurrences of $\tau_i$ in $C$. So if $i < p+1$, then
\begin{equation}\label{e:rhoi}
    \rho_i = \sum_{j=1}^{p+1}\sum_{k=1}^{p+1} z_{ijk} 
    = \sum_{k=1}^{p+1} \left(z_{i(p+1)k}+\sum_{\smash{j}=1}^{p} z_{ijk}  \right) 
    = \sum_{k=1}^{p+1} b_ke_i 
    = ne_i.
\end{equation}
    Similarly, 
    \[
    \rho_{p+1} 
    = \sum_{j=1}^{p+1}\sum_{k=1}^{p+1}\left(b_jb_k-\sum_{\ell=1}^{p}z_{\ell jk}\right)
    = \sum_{j=1}^{p+1} b_j\sum_{k=1}^{p+1}b_k-\sum_{\ell=1}^{p}\sum_{j=1}^{p+1}\sum_{k=1}^{p+1} z_{\ell jk}
    =n^2-\sum_{\ell=1}^pne_\ell
    =n^2-n(n-e_{p+1})
    = ne_{p+1},
    \]
    using \eref{e:rhoi}.
    Similar calculations show that $c_i = b_in$ and $\sigma_i = b_in$ for each $i \in [p+1]$. Thus, condition (i) of \dref{d:OR} is satisfied.
    
    Let $(i, j) \in [p+1]^2$. First, suppose that $j \leq i \leq p$. Then the number of symbols in cell $(i, j)$ of $C$ is
    \begin{equation}\label{eq:j_le_i}
        \sum_{k=1}^{p+1} z_{ijk} = \sum_{k=1}^{i} z'_{ijk} = e_ib_j,
    \end{equation}
    from \eref{e:sumk}. If $i < j \leq p$ then the number of symbols in cell $(i, j)$ of $C$ is 
    \begin{equation}\label{eq:j_>_i}
        \sum_{k=1}^{p+1} z_{ijk} = z_{ij(p+1)} = e_ib_j.
    \end{equation}
    If $i=p+1$, then the number of symbols in cell $(i, j)$ is
    \[
    \sum_{k=1}^{p+1} z_{(p+1)jk} = \sum_{k=1}^{p+1}\left(b_jb_k-\sum_{\ell=1}^{p}z_{\ell jk}\right) = nb_j-\sum_{\ell=1}^{p}\sum_{k=1}^{p+1} z_{\ell jk} = nb_j-\sum_{\ell=1}^{p}e_\ell b_j = e_{p+1}b_j,
    \]
    by \eqref{eq:j_le_i} and \eqref{eq:j_>_i}.
    Finally, if $i<j=p+1$ then the number of symbols in cell $(i, j)$ is
    \[
    \sum_{k=1}^{p+1} z_{i(p+1)k} = \sum_{k=1}^{p+1}\left(b_ke_i-\sum_{\smash{\ell}=1}^{p}z_{i\ell k}\right) = ne_i-\sum_{\ell=1}^{p}\sum_{k=1}^{p+1}z_{i\ell k} = ne_i-\sum_{\ell=1}^{p}b_\ell e_i = b_{p+1}e_i.
    \]
    Thus, condition (ii) of \dref{d:OR} is satisfied.
    
    The number of occurrences of $\tau_k$ in row $i < p+1$ is
\begin{equation}\label{eq:num_tau_k_row_i}
        \sum_{j=1}^{p+1} z_{ijk} =z_{i(p+1)k}+\sum_{j=1}^{p} z_{ijk} 
        = e_ib_k.
    \end{equation}
    The number of occurrences of $\tau_k$ in row $i=p+1$ is
    \[
    \sum_{j=1}^{p+1} z_{(p+1)jk} = \sum_{j=1}^{p+1} \left(b_jb_k-\sum_{\ell=1}^{p} z_{\ell jk}\right) = nb_k-\sum_{\ell=1}^{p}\sum_{j=1}^{p+1} z_{\ell jk} = nb_k-\sum_{\ell=1}^pe_\ell b_k=b_ke_{p+1},
    \]
    by \eqref{eq:num_tau_k_row_i}. Thus, condition (iii) of \dref{d:OR} is satisfied.
    
    The number of occurrences of $\tau_k$ in column $j$ is
    \[
    \sum_{i=1}^{p+1} z_{ijk}= z_{(p+1)jk}+\sum_{i=1}^{p} z_{ijk}
    = b_jb_k.
    \]
    Therefore, condition (iv) of \dref{d:OR} is satisfied and so $C$ is an outline rectangle. 
    By Theorem~\ref{thm:Hilton}, there is a Latin square $L$ of order $n$ and a triple $(u_1+\cdots+u_{p+1}, v_1+\cdots+v_{p+1}, w_1+\cdots+w_{p+1})$ of partitions of $n$ such that the submatrix of $L$ induced by the rows $X_i$ and columns $Y_j$ has exactly $z_{ijk}$ symbols from the set $Z_k$. 
    Here $X_i = \{u_{i-1}+1, \ldots, u_i\}$,  $Y_i = \{v_{i-1}+1, \ldots, v_i\}$ and $Z_k=\{w_{k-1}+1, \ldots, w_k\}$ for $i \in [p+1]$, where we interpret $u_0=v_0=w_0=0$. 
    Consider the filled Young diagram obtained from $L$ by removing all submatrices induced by rows $X_i$ and columns $Y_j$ where $p+1 \in \{i, j\}$ or $j > i$. This diagram has the same shape as $Y$. The inherited filling does not contain any repeated symbol in a row or column since $L$ is a Latin square. Furthermore, let $j \leq i \leq p$ and consider the submatrix with rows $X_i$ and $Y_j$. 
    The number of symbols in this submatrix from $\{a_{k-1}+1, \ldots, a_k\}$ is $z'_{ijk} = 0$ if $k > i$, meaning that no symbol is larger than $a_i$. 
    Therefore,  $Y$ has a Latin filling.
\end{proof}

From \tref{thm:alloc_implies_latin} we get the fact, as mentioned in Section~\ref{sec:introduction}, that \cref{conj:wide_implies_outline} is equivalent to \cref{conj:wpc}. \tref{thm:alloc_implies_latin} also provides us with an alternative proof of \tref{t:allocwide}, since it is known from \cite{chow2003} that all Latin Young diagrams are wide.


\section{Young diagrams with three row lengths}\label{sec:three}

In this section, we prove Conjecture~\ref{conj:wide_implies_outline} when the Young diagram has at most three distinct row lengths. This, together with  Theorem~\ref{thm:alloc_implies_latin}, will imply Theorem~\ref{thm:three}.

\begin{theorem}\label{thm:wide_implies_alloc}
    A wide Young diagram with three distinct row lengths has an allocation.
\end{theorem}

\begin{proof}
Let $Y$ be a wide Young diagram which consists of three row lengths $a_1$, $a_2$, and $a_3$ where $a_1 < a_2 < a_3$. Let $e_i$ be the number of rows of length $a_i$ for $i \in \{1, 2, 3\}$. 
The following inequalities, from \tref{thm:wideness_check}, are the wideness constraints for $Y$.
\begin{align}
    a_1&\ge e_1,\label{wideness:1}\\
    a_2&\ge e_1+e_2,\label{wideness:2}\\
    a_2e_2+a_1(a_1-e_2)&\ge a_1(e_1+e_2) \;\text{ 
    if   }\;e_2< a_1\le e_1+e_2,\label{wideness:3}\\
    a_3&\ge e_1+e_2+e_3,\label{wideness:4}
\end{align}
\begin{equation}\label{wideness:5}
    \begin{cases}
        a_3e_3+a_2(a_1-e_3)\ge a_1(e_1+e_2+e_3) & \text{if $e_3< a_1\le e_2+e_3$},\\
        a_3e_3+a_2e_2+a_1(a_1-e_2-e_3)\ge  a_1(e_1+e_2+e_3) & \text{if $e_2+e_3< a_1\le e_1+e_2+e_3$},
    \end{cases}
\end{equation}
and
\begin{equation}\label{wideness:6}
    \begin{cases}
        a_3e_3+a_2(a_2-e_3)\ge a_1e_1+a_2(e_2+e_3) & \text{if $e_3< a_2\le e_2+e_3$},\\
        a_3e_3+a_2e_2+a_1(a_2-e_2-e_3)\ge a_1e_1+a_2(e_2+e_3) & \text{if $e_2+e_3< a_2\le e_1+e_2+e_3$}.
    \end{cases}
\end{equation}

The wideness equations in \eref{wideness:3}, \eref{wideness:5}, and \eref{wideness:6} only hold when some condition is met. We note that for each of these inequalities, it is simple to verify that the inequality remains true if only the lower bound of the corresponding condition holds. For example, consider \eref{wideness:3}. If $a_1 > e_1+e_2$, then
\[
a_2e_2+a_1(a_1-e_2) \geq a_1e_2+a_1(a_1-e_2)=a_1^2 \geq a_1(e_1+e_2).
\]
So \eref{wideness:3} holds whenever $a_1 > e_2$, regardless of whether or not $a_1 \leq e_1+e_2$.
Thus, when we employ \eref{wideness:3}, \eref{wideness:5}, or \eref{wideness:6}, we will only check that the lower bound of the corresponding condition is met.

We need to show that there exist nonnegative integers $z_{ijk}$ for $1\le j,k\le i\le 3$ so that \eqref{e:sumk}, \eqref{e:sumj}, and \eqref{e:sumi} are satisfied. We define five variables $x=z_{211}$, $u=z_{311}$, $v=z_{312}$, $y=z_{321}$, and $w=z_{322}$, while the rest of the $z_{ijk}$ are determined from these variables by the identities  \eqref{e:sumk} and \eqref{e:sumj}. We will show that the values of $x$, $u$, $v$, $y$, and $w$ can be chosen so that the inequalities in \eqref{e:sumi} hold.

As a first step, we construct an allocation for the bottom two row blocks of $Y$. From \eqref{e:sumk}, we are forced to choose $z_{111} = a_1e_1$. By \eqref{e:sumk} and \eqref{e:sumj} with $p=2$, we have 
    \begin{equation}\label{eq:z212_z222}
        \begin{split}
            z_{212}&=z_{221}=a_1e_2-x,\\
            z_{222}&=b_2e_2-a_1e_2+x.
        \end{split}
    \end{equation}
    Since all $z_{ijk}$ are nonnegative, we require
    \begin{equation}\label{ineq:nonzero1}
       \max(0,e_2(a_1-b_2))\le x\le a_1e_2.
    \end{equation}
    
    There are four inequalities from \eref{e:sumi} with $p=2$ that must be satisfied:
    \begin{itemize}     
    \item  $j=k=1$: we need to satisfy $z_{111}+z_{211}=a_1e_1+x\le a_1^2$. So, we require
    \begin{equation}\label{ineq:x_less1}
        x\le a_1(a_1-e_1).
    \end{equation}

    \item  $j=1,k=2$: we need to satisfy $z_{212}=a_1e_2-x\le a_1b_2$. So, we require
    \begin{equation}\label{ineq:x_more}
        x\ge a_1(e_2-b_2).
    \end{equation}
    (Note that if $e_2\le b_2$, then \eqref{ineq:x_more} is redundant).
    \item  $j=2,k=1$: we obtain the same inequality as \eqref{ineq:x_more}.
    \item  $j=2,k=2$: we need to satisfy $z_{222} = b_2e_2-a_1e_2+x\le b_2^2$. So, we require
    \begin{equation}\label{ineq:x_less2}
        x\le a_1e_2-b_2(e_2-b_2).
    \end{equation}    
    \end{itemize}
    
    The constraints\eqref{ineq:nonzero1}--\eqref{ineq:x_less2} can be summarized in one expression:
    \begin{equation}\label{ineq:max_min}
        x\in I_x := [\max(0,e_2(a_1-b_2),a_1(e_2-b_2)), \min(a_1e_2,a_1(a_1-e_1),a_1e_2-b_2(e_2-b_2))].
    \end{equation}
    We claim that the interval $I_x$ is non-empty. For this, we need to show that every possible right endpoint of $I_x$ is at least as large as the left endpoint $\max(0,e_2(a_1-b_2),a_1(e_2-b_2))$. 
    Note that the value of the left endpoint of $I_x$ is equal to $0$, $e_2(a_1-b_2)$, or $a_1(e_2-b_2)$, when $\max(a_1 ,e_2 ,b_2)$ is equal to $b_2$, $a_1$, or $e_2$, respectively. 

    Clearly, $a_1e_2 \geq \max(0, e_2(a_1-b_2), a_1(e_2-b_2))$. Next, we verify that $a_1(a_1-e_1)$ and $a_1e_2-b_2(e_2-b_2)$ are both at least $\max(0, e_2(a_1-b_2), a_1(e_2-b_2))$. To do this, we consider the three possibilities for $\max(a_1 ,e_2 ,b_2)$ separately. 

    \begin{enumerate}
        \item $\max(0,e_2(a_1-b_2),a_1(e_2-b_2))=0$. By \eqref{wideness:1}, $a_1(a_1-e_1)\ge0$. Given that $\max(a_1, e_2, b_2) = b_2$, it follows that $a_1e_2-b_2(e_2-b_2)\ge0$.
        
        \item $\max(0,e_2(a_1-b_2),a_1(e_2-b_2))=e_2(a_1-b_2)$. In this case, $\max(a_1, e_2, b_2) = a_1$.
        The inequality $e_2(a_1-b_2)\le a_1(a_1-e_1)$ is equivalent to \eqref{wideness:3}. 
        The inequality $e_2(a_1-b_2)\le a_1e_2-b_2(e_2-b_2))$ is equivalent to $b_2^2 \geq 0$, which is clearly true.
        
        \item $\max(0,e_2(a_1-b_2),a_1(e_2-b_2))=a_1(e_2-b_2)$.  
        The inequality $a_1(e_2-b_2)\le a_1(a_1-e_1)$ is equivalent to $e_1+e_2\le a_2$, which is true by \eqref{wideness:2}.
        The inequality $a_1(e_2-b_2) \leq a_1e_2-b_2(e_2-b_2)$ is equivalent to $e_2 \leq a_2$, which is again true by \eqref{wideness:2}.
    \end{enumerate}
       
    Thus, $x$ can be chosen to satisfy the allocation requirements. To simplify the assignments in row block 3, we shall choose the minimal possible value for $x$, namely,
    \begin{equation}\label{eq:value_of_x}
		x = \begin{cases}
			0 & \text{if } \max (a_1 ,e_2 ,b_2) = b_2. \\
            e_2(a_1-b_2) & \text{if } \max (a_1 ,e_2 ,b_2) = a_1, \\
			a_1(e_2-b_2) & \text{if } \max (a_1 ,e_2 ,b_2) = e_2. \\
		\end{cases}
	\end{equation}

    We now extend our allocation to the top row block of $Y$. Recall that $u=z_{311}$, $v=z_{312}$, $y=z_{321}$, and $w=z_{322}$. So, by \eqref{e:sumk} and \eqref{e:sumj} we have 
    \begin{equation}\label{def:z333}
        \begin{split}
            z_{313}&=a_1e_3-u-v,\\
            z_{323}&=b_2e_3-y-w,\\
            z_{331}&=a_1e_3-u-y,\\
            z_{332}&=b_2e_3-v-w, \text{ and}\\
            z_{333}&=b_3e_3-b_2e_3-a_1e_3+u+v+y+w=(b_3-a_2)e_3+u+v+y+w.
        \end{split}
    \end{equation}
    Since we require all $z_{3jk}$ to be nonnegative, we obtain the following constraints:
    \begin{equation}\label{ineq:uvyw_nonneg}
        u,v,y,w\ge0,
    \end{equation}
    \begin{equation}\label{ineq:ub_sums_of_2}
    \begin{split}
            u+v&\le a_1e_3,\\
            y+w&\le b_2e_3,\\
            u+y&\le a_1e_3,\\
            v+w&\le b_2e_3,
    \end{split}
    \end{equation}
    and
    \begin{equation}\label{ineq:ub_sum_of_4}
        u+v+y+w\ge (a_2-b_3)e_3.
    \end{equation}

    The new constraints imposed by the nine inequalities in \eqref{e:sumi} with $p=3$ are:
    \begin{itemize}
        \item $j=1,k=1$: $z_{111}+z_{211}+z_{311}=a_1e_1+x+u\le a_1^2$. So, we require
        \begin{equation}\label{ineq:u_less}
            u\le (a_1-e_1)a_1-x.
        \end{equation}
        \item $j=1,k=2$: $z_{212}+z_{312}=a_1e_2-x+v\le a_1b_2$. So, we require
        \begin{equation}\label{ineq:v_less}
            v\le x-a_1(e_2-b_2).
        \end{equation}
        \item $j=1,k=3$: $z_{313}=a_1e_3-u-v\le a_1b_3$. So, we require
        \begin{equation}\label{ineq:u+v_more}
            u+v\ge a_1(e_3-b_3).
        \end{equation}
        \item $j=2,k=1$: $z_{221}+z_{321}=a_1e_2-x+y\le a_1b_2$. So, we require
        \begin{equation}\label{ineq:y_less}
            y\le x- a_1(e_2-b_2).
        \end{equation}
        \item $j=2,k=2$: $z_{222}+z_{322}=b_2e_2-a_1e_2+x+w\le b_2^2$. So, we require
        \begin{equation}\label{ineq:w_less}
            w\le a_1e_2-b_2(e_2-b_2)-x.
        \end{equation}
        \item $j=2,k=3$: $z_{323}=b_2e_3-y-w\le b_2b_3$. So, we require
        \begin{equation}\label{ineq:y+w_more}
            y+w\ge b_2(e_3-b_3).
        \end{equation}
        \item $j=3,k=1$: $z_{331}=a_1e_3-u-y\le a_1b_3$. So, we require
        \begin{equation}\label{ineq:u+y_more}
            u+y\ge a_1(e_3-b_3).
        \end{equation}
        \item $j=3,k=2$: $z_{332}=b_2e_3-v-w\le b_2b_3$. So, we require
        \begin{equation}\label{ineq:v+w_more}
            v+w\ge b_2(e_3-b_3).
        \end{equation}
        \item $j=3,k=3$: $z_{333}=(b_3-a_2)e_3+u+v+y+w\le b_3^2$. So, we require
        \begin{equation}\label{ineq:all_less}
            u+v+y+w\le b_3^2+(a_2-b_3)e_3.
        \end{equation}
    \end{itemize}
    We need to show that there are nonnegative integers $u$, $v$, $y$, $w$ that satisfy all constraints \eref{ineq:uvyw_nonneg}--\eref{ineq:all_less}.
    
    The constraints define lower and upper bounds for the values of $u$, $v$, $y$, and $w$ and for the sums $u+v$, $u+y$, $y+w$, $v+w$, and $u+v+y+w$. Our first goal is to show that the intervals defined by these lower and upper bounds are nonempty. 

    From  \eqref{ineq:uvyw_nonneg}, \eqref{ineq:u_less},\eqref{ineq:v_less},\eqref{ineq:y_less} and \eqref{ineq:w_less}, we have the requirements
    \begin{equation}\label{eq:uvyw_in_intervals}
        \begin{split}
            u&\in I_u:=[0,a_1(a_1-e_1)-x],\\
            v,y&\in I_v=I_y:=[0,x - a_1(e_2-b_2)], \text{ and} \\
            w&\in I_w:=[0,a_1e_2-b_2(e_2-b_2)-x].
        \end{split}
    \end{equation}
    By \eqref{ineq:x_less1}, \eqref{ineq:x_more}, and \eqref{ineq:x_less2} we have:
    \begin{claim}\label{claim:interval_nonempty_singles}
        The intervals $I_u,I_v,I_y,$ and $I_w$ are not empty.
    \end{claim}

    Let $I_{u+v}=I_{u+y}=[a_1(e_3-b_3),a_1e_3]$ and $I_{y+w}=I_{v+w}=[b_2(e_3-b_3),b_2e_3]$. By \eqref{ineq:ub_sums_of_2}, \eqref{ineq:u+v_more}, \eqref{ineq:y+w_more}, \eqref{ineq:u+y_more}, and \eqref{ineq:v+w_more}, we require
\begin{equation}\label{eq:pairs_intervals}
        u+v\in I_{u+v},\  
        u+y\in I_{u+y}, \ 
        v+w\in I_{v+w} \text{ and }
        y+w\in I_{y+w}.
   \end{equation}

\begin{notation}
        For two intervals $I=[a,b]$ and $J=[c,d]$, let $I\oplus J:=[a+c,b+d]$.
    \end{notation}

    Let $\I_{u+v} = I_{u+v}\cap (I_{u}\oplus I_{v})$, $\I_{u+y} = I_{u+y}\cap (I_{u}\oplus I_{y})$, $\I_{v+w} = I_{v+w}\cap (I_{v}\oplus I_{w})$, and $\I_{y+w} = I_{y+w}\cap (I_{y}\oplus I_{w})$.
    By \eqref{eq:uvyw_in_intervals} and \eqref{eq:pairs_intervals}, we require  \begin{equation}\label{eq:pairs_intersection_intervals}
        u+v\in \I_{u+v},\
        u+y\in \I_{u+y},\
        v+w\in \I_{v+w}  \text{ and }
        y+w\in \I_{y+w}.
\end{equation}
    Note that $\I_{u+v}=\I_{u+y}$ and $\I_{v+w}=\I_{y+w}$. So, from now on, we shall only use $\I_{u+y}$ and $\I_{y+w}$.
    
    \begin{claim}\label{claim:interval_nonempty_pairs}
        The intervals $\I_{u+y}$ and $\I_{y+w}$ are not empty.        
    \end{claim}
    \begin{proof} 
        By definition of $I_{u+y}$ and $I_{u}\oplus I_{y}$,    \begin{equation}\label{ineq:u+v_interval}
            \I_{u+y}=[\max(0,a_1(e_3-b_3)),\min(a_1e_3,a_1(a_2-e_1-e_2))].
        \end{equation}
        To show that this is indeed a non-empty interval, we show that each possible left endpoint is at most each possible right endpoint. The required inequalities involving $a_1e_3$ are obviously true. In addition, $a_1(a_2-e_1-e_2)\ge0$ by \eqref{wideness:2} and $a_1(e_3-b_3)\le a_1(a_2-e_1-e_2)$ by \eqref{wideness:4}.

        Next we prove the claim regarding $\I_{y+w}$. From the definitions of $I_{y+w}$ and $I_y \oplus I_w$,
\begin{equation}\label{ineq:y+w_interval}
        \I_{y+w}=[\max(0,b_2(e_3-b_3)),\min(b_2e_3,b_2(a_2-e_2))].
        \end{equation}
        The fact that $b_2e_3$ is at least as large as both possible left endpoints is obvious. Furthermore, $b_2(a_2-e_2)>0$ by \eqref{wideness:2} and $b_2(e_3-b_3)<b_2(a_2-e_2)$ by \eqref{wideness:4}.
    \end{proof}

    Define $J_1:=\I_{u+y}\oplus\I_{y+w}$. That is,
    \begin{equation}
        \begin{split}
            J_1 &= [\max(0,a_1(e_3-b_3))+\max(0,b_2(e_3-b_3)),\min(a_1e_3,a_1(a_2-e_1-e_2))+\min(b_2e_3,b_2(a_2-e_2))]\\
            &=[\max(0,a_2(e_3-b_3)),a_1\min(e_3,a_2-e_1-e_2)+b_2\min(e_3,a_2-e_2)].
        \end{split}            
    \end{equation} 

    Let $J_2$ be the interval whose endpoints are the lower and upper bounds for $u+v+y+w$ given in \eqref{ineq:ub_sum_of_4} and \eqref{ineq:all_less}. That is,
    \begin{equation}\label{def:J2}
        J_2=[(a_2-b_3)e_3,(a_2-b_3)e_3+b_3^2].
    \end{equation}

    Let $\I_{u+v+y+w}=J_1\cap J_2$. The constraints for $u$, $v$, $y$, and $w$ require 
    \begin{equation}\label{eq:uvyw_interval}
        u+v+y+w\in \I_{u+v+y+w}.
    \end{equation}
    
    \begin{claim}\label{claim:interval_nonempty_all4}
        The set $\I_{u+v+y+w}$ is a nonempty interval.        
    \end{claim}
    \begin{proof}
Let 
\begin{equation}\label{ineq:all_combined_lower}
\ell = \max\big(0,a_2(e_3-b_3),(a_2-b_3)e_3\big) = \max\big(0,a_2e_3-b_3\min(a_2,e_3)\big)
\end{equation}
and 
\[
m=\min\Big(\min(a_1e_3,a_1(a_2-e_1-e_2))+\min(b_2e_3,b_2(a_2-e_2)),\ (a_2-b_3)e_3+b_3^2\Big).
\]
To prove the claim, it suffices to prove that $\ell \leq m$.

Let $\alpha = a_1e_3$, $\beta=a_1(a_2-e_1-e_2)$, $\gamma=b_2e_3$, and $\delta = b_2(a_2-e_2)$. So,  \begin{equation}\label{ineq:all_combined_upper_simple}
 m=\min\Big(\min(\alpha,\beta)+\min(\gamma,\delta),\ a_2e_3- b_3(e_3-b_3)\Big).
\end{equation}
        
Since $\min(\alpha,\beta)=\alpha$ if and only if $a_2\ge e_1+e_2+e_3$ and $\min(\gamma,\delta)=\gamma$ if and only if $a_2\ge e_2+e_3$,  $\min(\alpha,\beta)=\alpha$ implies $\min(\gamma,\delta)=\gamma$. Thus, we can consider the following three cases:

\bigskip\noindent{\bf Case (i):}
$\min(\alpha,\beta)=\alpha$ and $\min(\gamma,\delta)=\gamma$.

 As mentioned above, $a_2\ge e_1+e_2+e_3$, so $\min(a_2,e_3)=e_3$. Combining \eqref{ineq:all_combined_lower} and \eqref{ineq:all_combined_upper_simple}, the requirement for $\ell\le m$ is          
 \begin{equation*}
    \max(0,(a_2-b_3)e_3)\le \min(a_2e_3,\ a_2e_3-b_3(e_3-b_3)).
 \end{equation*}
It is obvious that $\min(a_2e_3,a_2e_3-b_3(e_3-b_3)) \geq (a_2-b_3)e_3$, so it 
suffices to show that $a_2e_3\ge b_3(e_3-b_3)$ when $e_3\ge b_3$. However, this is immediate from $a_2\ge e_3$.

\bigskip\noindent{\bf Case (ii):}
$\min(\alpha,\beta)=\beta$ and $\min(\gamma,\delta)=\gamma$.
            
In this case, $a_2\ge e_2+e_3$ so, again, $\min(a_2,e_3)=e_3$. Combining \eqref{ineq:all_combined_lower} and \eqref{ineq:all_combined_upper_simple}, the requirement for $\ell\le m$ is
\begin{equation*}
\max(0,(a_2-b_3)e_3)\le\min\Big(a_1(a_2-e_1-e_2)+b_2e_3,\ a_2e_3-b_3(e_3-b_3)\Big).
\end{equation*} 
 We first verify that $\min(a_1(a_2-e_1-e_2)+b_2e_3,a_2e_3-b_3(e_3-b_3)) \geq 0$. By \eqref{wideness:2}, $a_1(a_2-e_1-e_2)+b_2e_3 \geq 0$. 
 Also, $a_2e_3\ge b_3(e_3-b_3)$ by the same argument as in Case (i).
 Next, we note that $(a_2-b_3)e_3\le a_1(a_2-e_1-e_2)+b_2e_3$ by \eqref{wideness:6}.
 Finally, the fact that $(a_2-b_3)e_3< a_2e_3-b_3(e_3-b_3)$ is obvious. 

\bigskip\noindent{\bf Case (iii):}
$\min(\alpha,\beta)=\beta$ and $\min(\gamma,\delta)=\delta$.
            
 Here, $a_2<e_2+e_3$ and we need to satisfy
\begin{equation}\label{ineq:all_interval_III}
 \max(0,a_2e_3-b_3\min(a_2,e_3))\le\min\Big(a_1(a_2-e_1-e_2)+b_2(a_2-e_2),\ a_2e_3-b_3(e_3-b_3)\Big).
\end{equation}
First, consider when $a_2e_3<b_3\min(a_2,e_3)$. By \eref{wideness:2}, $a_1(a_2-e_1-e_2)+b_2(a_2-e_2) \geq 0$. 
If $e_3\le a_2$ then
$a_2e_3\ge b_3(e_3-b_3)$
for the same reason as in Case (i). On the other hand, if $a_2<e_3$, then from $a_2e_3<b_3\min(a_2,e_3)$, it follows that $e_3 < b_3$, and $a_2e_3\ge0>b_3(e_3-b_3)$.
            
Now, assume that $a_2e_3\ge b_3\min(a_2,e_3)$. To verify \eref{ineq:all_interval_III}, we have four subcases to consider, depending on the values of $\min(a_2, e_3)$ and $\min(a_1(a_2-e_1-e_2)+b_2(a_2-e_2),a_2e_3-b_3(e_3-b_3))$.

First, assume that $a_1(a_2-e_1-e_2)+b_2(a_2-e_2)\le a_2e_3-b_3(e_3-b_3)$.
\begin{itemize}
\item Suppose that $a_2> e_3$. Then we need to prove that $a_2e_3-b_3e_3\le a_1(a_2-e_1-e_2)+b_2(a_2-e_2)$. By \eqref{wideness:6},
    $b_3e_3\ge a_1e_1+a_2(e_2+e_3)-a_2^2$. Thus,
     \[
     a_2e_3-b_3e_3 \leq a_2e_3-(a_1e_1+a_2(e_2+e_3)-a_2^2)= a_1(a_2-e_1-e_2)+b_2(a_2-e_2).
    \]
                
 \item Suppose that $a_2\le e_3$. We need to prove that $a_2e_3-a_2b_3\le a_1(a_2-e_1-e_2)+b_2(a_2-e_2)$, which is equivalent to 
$a_2(e_2+e_3)\le a_2a_3-a_1e_1$. By \eqref{wideness:4} and the fact that $a_1 \leq a_2$, we have
\[
a_2a_3-a_1e_1 \geq a_2(a_3-e_1) \geq a_2(e_2+e_3).
\]
\end{itemize}
Finally, assume that $a_1(a_2-e_1-e_2)+b_2(a_2-e_2)>a_2e_3-b_3(e_3-b_3)$.
\begin{itemize}
\item Suppose that $a_2>e_3$. The required inequality in this case is $a_2e_3-b_3e_3\le a_2e_3-b_3(e_3-b_3)$, which is easily seen to be true. 
 \item Suppose that $a_2\le e_3$. We need to prove $a_2e_3-a_2b_3\le a_2e_3-b_3(e_3-b_3)$. This follows easily using the fact that $e_3<a_3$, by \eqref{wideness:4}. \qedhere
\end{itemize}
        
\end{proof}

\begin{observation}\label{obs:equiv_constraints}
 The set of constraints \eqref{ineq:uvyw_nonneg}--\eqref{ineq:all_less} is equivalent to the requirements \eqref{eq:uvyw_in_intervals}, \eqref{eq:pairs_intersection_intervals}, and \eqref{eq:uvyw_interval}.
\end{observation}

Since the constraints \eqref{eq:uvyw_in_intervals}, \eqref{eq:pairs_intersection_intervals}, and \eqref{eq:uvyw_interval} are symmetric in $v$ and $y$, it makes sense to look for an allocation where $y=v$. From now on, we will impose this extra condition.

\begin{notation}
    For an interval $I$, we write $\min(I)$ for the left endpoint of $I$ and $\max(I)$ for the right endpoint of $I$.
\end{notation}


Given $s\in \I_{u+v+y+w}$, we can find $s_1\in \I_{u+y}$ and $s_2\in \I_{y+w}$ such that $s_1+s_2 = s$. We want to show that we can choose $s,s_1$, and $s_2$ so that there are numbers $u$, $y$, and $w$ satisfying \eqref{eq:uvyw_in_intervals} such that $s_1=u+y$ and $s_2=y+w$.
Given $s_2\in \I_{y+w}$, we know that $s_2\in I_y\oplus I_w$ (since $\I_{y+w}\subseteq I_y\oplus I_w$). So, there exist $y\in I_y$ and $w\in I_w$ such that $y+w=s_2$. Set $u=s_1-y$. If $u\in I_u$ for any such choice of $y$ and $w$, then we are done. So, we will assume that $u\not\in I_u$ for any valid choice of $y$ and $w$. 
From now on, we will assume that we have chosen $w$ and $y$ optimally, so that the distance $\max(0, \min(I_u)-u)+\max(0, u-\max(I_u))$ from $u$ to $I_u$ is minimised. We will then reach a contradiction, implying that there is a choice of $y$ and $w$ such that $u \in I_u$. Since $u\not\in I_u$, there are two cases to consider. 

\bigskip\noindent{\large{\bf Case (i):\ } $u<0$. }

\begin{proof}
    In this case, $y=s_1-u>s_1$. The proof proceeds via a sequence of claims.

\begin{claim}\label{claim:w=max}
    $w=\max(I_w)$.
\end{claim}
\begin{proof}
    Suppose, for a contradiction, that $w<\max(I_w)$. Then we can decrease $y$ by 1 (which we can do, since $y>s_1 \geq 0 = \min(I_y)$) and increase $u$ and $w$ by 1, which preserves the values of $s$, $s_1$, and $s_2$. This means that our initial choice of $y$ and $w$ did not minimise the distance from $u$ to $I_u$, which is a contradiction.
\end{proof}

\begin{claim}\label{claim:J1_no_singleton_1}
    If $s_1<\max(\I_{u+y})$ and $s_2>\min(\I_{y+w})$, then $|\I_{u+v+y+w}|>1$. 
\end{claim}
\begin{proof}
    Suppose that $s_1<\max(\I_{u+y})$ and $s_2>\min(\I_{y+w})$. Since $s_1$ is not the right endpoint of $\I_{u+y}$ and $s_2$ is not the left endpoint of $\I_{y+w}$, $s=s_1+s_2$ is not an endpoint of $J_1=\I_{u+y}\oplus\I_{y+w}$. So, $|J_1|>1$. By \eqref{def:J2}, we also have $|J_2|>1$. By definition, $s\in J_1\cap J_2$. If $|J_1\cap J_2|=1$, then $s$ must be an endpoint of both $J_1$ and $J_2$. But we know that $s$ is not an endpoint of $J_1$. Thus, $|J_1\cap J_2|=|\I_{u+v+y+w}|>1$.
\end{proof}

\begin{claim}\label{claim:s1=max_or_s2=min}
    $s_1=\max(\I_{u+y})$ or $s_2=\min(\I_{y+w})$.
\end{claim}
\begin{proof}
    Suppose, for a contradiction, that $s_1<\max(\I_{u+y})$ and $s_2>\min(\I_{y+w})$. By Claim~\ref{claim:J1_no_singleton_1}, $s>\min(\I_{u+v+y+w})$ or $s<\max(\I_{u+v+y+w})$ (or both). 
    If $s>\min(\I_{u+v+y+w})$, then we can decrease $y$ and increase $u$. This will decrease $s$ and $s_2$, which is allowed, and will decrease the distance from $u$ to $I_u$, which is a contradiction. Similarly, if $s<\max(\I_{u+v+y+w})$, then we can increase $u$, another contradiction. 
\end{proof}

\begin{claim}\label{claim:a2>e1+e2+e3}
    If $s_1=\max(\I_{u+y})$, then $a_2\ge e_1+e_2+e_3$.
\end{claim}
\begin{proof}
    Suppose that $s_1=\max(\I_{u+y})$. Then $\max(\I_{u+y}) = s_1 <y\leq\max(\I_{y})$. 
    Applying \eqref{eq:uvyw_in_intervals} and \eqref{ineq:u+v_interval}, this gives 
    $$\min(a_1e_3,a_1(a_2-e_1-e_2))<x - a_1(e_2-b_2).$$
    If $a_2<e_1+e_2+e_3$, then $\min(a_1e_3,a_1(a_2-e_1-e_2))=a_1(a_2-e_1-e_2)$, and
    $$a_1(a_2-e_1-e_2)<x - a_1(e_2-b_2)\le a_1(a_1-e_1) - a_1(e_2-b_2),$$ 
    by \eqref{ineq:max_min}.This implies that
    $$a_2-e_1-e_2<a_1-e_1-e_2+b_2=a_2-e_1-e_2,$$
    a contradiction. Thus, $a_2\ge e_1+e_2+e_3$.
\end{proof}

\begin{claim}\label{claim:s1=max}
    $s_1<\max(\I_{u+y})$.
\end{claim}
\begin{proof}
    Suppose, for a contradiction, that $s_1=\max(\I_{u+y})$. We show that this implies that $s_2>\max(\I_{y+w})$, contradicting the choice of $s_2$.
    By \clref{claim:a2>e1+e2+e3}, $y>s_1=\max(\I_{u+y})=\min(a_1e_3,a_1(a_2-e_1-e_2))=a_1e_3$. So, by Claim~\ref{claim:w=max}, 
    \begin{equation}\label{e:s2ub}
    s_2=y+w>a_1e_3+\max(I_w)=a_1e_3+a_1e_2-b_2(e_2-b_2)-x.
    \end{equation}
    On the other hand, by \clref{claim:a2>e1+e2+e3}, 
    \begin{equation}\label{e:s2lb}
    s_2 \leq \max(\I_{y+w})=\min(b_2e_3,b_2(a_2-e_2))=b_2e_3. 
    \end{equation}
    So, to reach our contradiction we have to show that
    \begin{equation*}
        a_1e_3+a_1e_2-b_2(e_2-b_2)-x\ge b_2e_3.
    \end{equation*}
    This inequality is equivalent to
    \begin{equation}\label{exp:s2>maxIv+w}
        x\le (a_1-b_2)(e_2+e_3)+b_2^2.
    \end{equation}

    First consider when $a_1\le b_2$. By \eqref{ineq:max_min}, it suffices to prove the inequality
\begin{equation}\label{e:stronger}a_1(a_1-e_1)\le (a_1-b_2)(e_2+e_3)+b_2^2,\end{equation}
    which is equivalent to $a_1(a_2-e_2-e_3)-a_1e_1\le b_2(a_2-e_2-e_3)$. This holds, since $a_1\le b_2$.

    Now consider when $a_1>b_2$. By \eqref{ineq:max_min}, it suffices to prove the inequality
    \[a_1e_2-b_2(e_2-b_2)\le (a_1-b_2)(e_2+e_3)+b_2^2,\]
    which is equivalent to $b_2e_3\le a_1e_3$. This holds, since $a_1 > b_2$.
\end{proof}

\begin{claim}\label{claim:s2=min}
    $s_2>\min(\I_{y+w})$.
\end{claim}
\begin{proof}
    We will show that if $s_2=\min(\I_{y+w})$, then $s_1<\min(\I_{u+y})$, which contradicts the definition of $s_1$.
    By Claim~\ref{claim:w=max},
    $$s_1<y=s_2-w=\min(\I_{y+w})-\max(I_w)=\max(0,b_2(e_3-b_3))-(a_1e_2-b_2(e_2-b_2)-x).$$ 
    So, to reach our contradiction it will suffice to prove
\begin{equation}\label{ineq:max-max<max}
        \max(0,b_2(e_3-b_3))-a_1e_2+b_2(e_2-b_2)+x\le \max(0,a_1(e_3-b_3)).
    \end{equation}
    If $e_3\le b_3$, then \eqref{ineq:max-max<max} is equivalent to $x\le a_1e_2-b_2(e_2-b_2)$, which is true, by \eqref{ineq:x_less2}.
    If $e_3>b_3$, then \eqref{ineq:max-max<max} becomes $b_2(e_3-b_3)-a_1e_2+b_2(e_2-b_2)+x\le a_1(e_3-b_3)$, which is equivalent to 
    \begin{equation}\label{ineq:ub_for_x}
        x\le (a_1-b_2)(e_2+e_3-b_3)+b_2^2.
    \end{equation}
    To prove \eref{ineq:ub_for_x}, we consider two cases, depending on whether or not $a_1 \geq b_2$. 

    First, suppose that $a_1\ge b_2$. By \eqref{ineq:max_min}, it suffices to prove the inequality $$a_1e_2-b_2(e_2-b_2)\le (a_1-b_2)(e_2+e_3-b_3)+b_2^2,$$ which is equivalent to $0\le(a_1-b_2)(e_3-b_3)$. This is true, since $e_3>b_3$ and $a_1 \geq b_2$.
    
    Finally, consider when $a_1< b_2$. By \eqref{ineq:max_min}, it suffices to note that the inequality $$a_1(a_1-e_1)\le (a_1-b_2)(e_2+e_3-b_3)+b_2^2,$$ is implied by \eref{e:stronger}. 
\end{proof}

Claims \ref{claim:s1=max} and \ref{claim:s2=min}  contradict \clref{claim:s1=max_or_s2=min}. Thus, we have reached our desired contradiction in Case (i). 
\end{proof}


\bigskip\noindent{\large{\bf  Case (ii):\ } $u>\max(I_u)$. }

\begin{proof}
The proof again proceeds via a sequence of claims.

\begin{claim}\label{claim:y<max}
    $y<\max(I_y)$.
\end{claim}
\begin{proof}
    If $y=\max(I_y)$, then $s_1=u+y>\max(I_u)+\max(I_y)$. So, $s_1\not\in I_u\oplus I_y$, contrary to the choice of $s_1$.
\end{proof}
\begin{claim}\label{claim:u=maxIu}
    $w=0$.
\end{claim}
\begin{proof}
    If $w>0$, then, by Claim~\ref{claim:y<max} we can increase $y$ and decrease $w$ and $u$. This decreases the distance from $u$ to $I_u$, a contradiction.
\end{proof}

\begin{claim}\label{claim:J1_no_singleton_2}
    If $s_1>\min(\I_{u+y})$ and $s_2<\max(\I_{y+w})$, then $|\I_{u+v+y+w}|>1$. 
\end{claim}
\begin{proof}
    The proof is the same as that of Claim~\ref{claim:J1_no_singleton_1}.
\end{proof}

\begin{claim}\label{claim:s1=min_or_s2=max}
    $s_1=\min(\I_{u+y})$ or $s_2=\max(\I_{y+w})$. 
\end{claim}
\begin{proof}
    Suppose, for a contradiction, that $s_1>\min(\I_{u+y})$ and $s_2<\max(\I_{y+w})$. By Claim~\ref{claim:J1_no_singleton_2}, $s>\min(\I_{u+v+y+w})$ or $s<\max(\I_{u+v+y+w})$ (or both). 
    If $s>\min(\I_{u+v+y+w})$, then we can decrease $u$ (and thus also $s$ and $s_1$), which is a contradiction. 
    If $s<\max(\I_{u+v+y+w})$, then we can increase $y$ and decrease $u$, another contradiction. 
\end{proof}

\begin{claim}\label{claim:if_s2=max:1}
    If $s_2=\max(\I_{y+w})$, then $a_2\ge e_2+e_3$.
\end{claim}
\begin{proof}
    We have $\max(\I_{y+w})=s_2=w+y=y< \max(I_y)$, by Claim~\ref{claim:y<max} and Claim~\ref{claim:u=maxIu}.
    So, $\min(b_2e_3,b_2(a_2-e_2))< x-a_1(e_2-b_2)$. 
    Suppose, for a contradiction, that $a_2< e_2+e_3$. Then $\min(b_2e_3,b_2(a_2-e_2))=b_2(a_2-e_2)$.
    So, by \eqref{ineq:max_min}, $b_2(a_2-e_2)<a_1e_2-b_2(e_2-b_2)-a_1(e_2-b_2) = b_2(a_2-e_2)$, a contradiction. 
\end{proof}

When $b_2 \geq a_1$, we can easily arrive at a contradiction to Claim~\ref{claim:s1=min_or_s2=max}.

\begin{claim}\label{claim:if_s1=min:1}
    If $b_2\ge a_1$, then $s_1>\min(\I_{u+y})$.
\end{claim}
\begin{proof}
    Suppose that $s_1=\min(\I_{u+y})$. We will prove that $s_2<\min(\I_{y+w})$, a contradiction. Since $s_1=u+y\ge u > 0$ and $s_1=\min(\I_{u+y})$, it follows that $\min(\I_{u+y})=a_1(e_3-b_3)>0$. So, $e_3>b_3$ and thus $\min(\I_{y+w})=b_2(e_3-b_3)$.
    Thus, by Claim~\ref{claim:u=maxIu}, $$s_2=y+w=y=s_1-u<s_1=\min(\I_{u+y})=a_1(e_3-b_3)\le b_2(e_3-b_3)=\min(\I_{y+w}),$$
    our desired contradiction.
\end{proof}

\begin{claim}\label{claim:if_s2=max:2}
    If $b_2\ge a_1$, then $s_2<\max(\I_{y+w})$.
\end{claim}
\begin{proof}
Suppose that $s_2=\max(\I_{y+w})$. We will prove that $s_1>\max(\I_{u+y})$, a contradiction. 
     First note that by Claim~\ref{claim:if_s2=max:1}, $\max(\I_{y+w})=b_2e_3$. Thus, by Claim~\ref{claim:u=maxIu}, $$s_1=y+u>y=s_2=\max(\I_{y+w})=b_2e_3\ge a_1e_3\ge \max(\I_{u+y}),$$
     our desired contradiction.
\end{proof}
So, if $a_1 \leq b_2$, then Claims~\ref{claim:if_s1=min:1}~and~\ref{claim:if_s2=max:2} contradict Claim~\ref{claim:s1=min_or_s2=max}. 
Henceforth, we may assume that
$a_1>b_2$. 

\medskip

We next introduce some new notation. For $i, j, k=1,2,3$, define
\[
x_{ijk} = b_jb_k-\sum_{\ell = \max(j, k)}^{i} z_{ijk}.
\]
Then $x_{ijk}$ represents the difference between the total number $b_jb_k$ of symbols from symbol block $k$ that could appear in column block $j$ and the number of such symbols that have been allocated to column block $j$.
For $j=1,2,3$, define
\begin{equation}\label{e:rhoj}
    \rho_j = x_{3j1}+x_{3j2}+x_{3j3}.
\end{equation}
So $\rho_j$ represents the difference between the number of symbols that could have been allocated to column block $j$ and the symbols that were allocated to column block $j$. Thus, we can determine the following explicit values for each $\rho_j$:
\begin{align}
    \rho_1&=a_1(a_3-e_1-e_2-e_3),\label{eq:rho_1}\\
    \rho_2&=b_2(a_3-e_2-e_3),\label{eq:rho_2}\\
    \rho_3&=b_3(a_3-e_3).\label{eq:rho_3}
\end{align}
It will be convenient for us to have an explicit list of the values of $x_{ijk}$ whenever $i = 2, 3$ and $j, k = 1, 2, 3$. 
\begin{align}\label{xijk_values}
        x_{211} &= a_1(a_1-e_1)-z_{211}= a_1(a_1-e_1)-x,\nonumber\\
        x_{212} &= x_{221} = a_1b_2-z_{212}= a_1b_2-(a_1e_2-x),\nonumber\\
        x_{222} &= b_2^2-z_{222}=b_2^2-(b_2e_2-a_1e_2+x),\nonumber\\
        x_{213} &= x_{231} = a_1b_3,\nonumber\\
        x_{223} &= x_{232} = b_2b_3,\nonumber\\
        x_{233} &= b_3^2,\nonumber\\
        x_{311}&=x_{211}-z_{311}=a_1(a_1-e_1)-x-u,\nonumber\\
        x_{312}&=x_{212}-z_{312}=a_1b_2-(a_1e_2-x)-y=a_1(b_2-e_2)+x-y,\\
        x_{313}&=x_{213}-z_{313}=a_1b_3-(a_1e_3-u-y)=a_1(b_3-e_3)+u+y,\nonumber\\
        x_{321}&=x_{221}-z_{321}=a_1b_2-(a_1e_2-x)-y=a_1(b_2-e_2)+x-y=x_{312},\nonumber\\
        x_{322}&=x_{222}-z_{322}=b_2^2-(b_2e_2-a_1e_2+x)-w=b_2(b_2-e_2)+a_1e_2-x-w,\nonumber\\
        x_{323}&=x_{223}-z_{323}=b_2b_3-(b_2e_3-y-w)=b_2(b_3-e_3)+y+w,\nonumber\\
        x_{331}&=x_{231}-z_{331}=a_1b_3-(a_1e_3-u-y)=a_1(b_3-e_3)+u+y=x_{313},\nonumber\\
        x_{332}&=x_{232}-z_{332}=b_2b_3-(b_2e_3-y-w)=b_2(b_3-e_3)+y+w=x_{323},\nonumber\\
        x_{333}&=x_{233}-z_{333}=b_3^2-((b_3-a_2)e_3+u+y+y+w)=b_3^2+(a_2-b_3)e_3-(u+y+y+w).\nonumber
\end{align}

\begin{claim}\label{claim:s_1=min}
    If $s_1=\min(\I_{u+y})$, then $s_1=a_1(e_3-b_3)$,  $x_{311}<0$, and $x_{313}=x_{331}=0$.
\end{claim}
\begin{proof}
Suppose that $s_1=\min(\I_{u+y})$. By \eqref{ineq:u+v_interval}, $s_1 = \max(0, a_1(e_3-b_3))$. Since $s_1 = u+y > \max(I_u)+y \geq 0$, it follows that $s_1 = a_1(e_3-b_3)$.
From \eqref{xijk_values} and \eqref{eq:uvyw_in_intervals}, $x_{311}=a_1(a_1-e_1)-x-u=\max(I_u)-u<0$, since $u>\max(I_u)$. Also
from \eqref{xijk_values}, $x_{313}=x_{331}=u+y-a_1(e_3-b_3)=s_1-s_1=0$. 
\end{proof}

\begin{claim}\label{claim:s_2=max}
    If $s_2=\max(\I_{y+w})$, then $s_2=b_2e_3$ and $x_{323}=x_{223}=b_2b_3$.
\end{claim}
\begin{proof}
Suppose that $s_2=\max(\I_{y+w})$. By \eqref{ineq:y+w_interval} and Claim~\ref{claim:if_s2=max:1}, $y+w=s_2=b_2e_3$. So  $z_{323}=0$, by \eqref{def:z333}. Hence, $x_{323}=x_{223}=b_2b_3$, by \eqref{xijk_values}.
\end{proof}

By Claim~\ref{claim:s1=min_or_s2=max}, it suffices to consider the following three subcases: 

\bigskip
\noindent{\bf Subcase (i):\ } $a_1>b_2$ and $s_1=\min(\I_{u+y})$ and $s_2=\max(\I_{y+w})$.

    By \eqref{e:rhoj} and Claim~\ref{claim:s_1=min},
    \begin{equation}\label{eq:rho1<x312}
        \rho_1  <x_{312}.
    \end{equation}
    Now, $z_{322}=w=0$, by Claim~\ref{claim:u=maxIu}, so \eqref{eq:z212_z222} and \eqref{xijk_values} imply that \begin{equation}\label{eq:x322_x222}
        x_{322}=x_{222}=b_2^2+e_2(a_1-b_2)-x.
    \end{equation}
    By \eqref{e:rhoj},  \eqref{eq:rho1<x312}, \eqref{eq:x322_x222}, and Claim~\ref{claim:s_2=max}, $\rho_2=x_{321}+x_{322}+x_{323}>\rho_1+x_{222}+x_{223}$. Thus,
    \begin{equation}\label{eq:rho2-rho1}
        \rho_2-\rho_1>x_{222}+x_{223}.
    \end{equation}
    By \eref{eq:rho2-rho1}, \eqref{eq:x322_x222}, \eqref{eq:rho_1},\eqref{eq:rho_2} and Claim~\ref{claim:s_2=max},
    \begin{equation}\label{ineq:rho2-rho1}
        b_2(a_3-e_2-e_3)-a_1(a_3-e_1-e_2-e_3)>b_2^2+e_2(a_1-b_2)-x+b_2b_3.
    \end{equation}
    We will show that the inequality \eqref{ineq:rho2-rho1} is false. We will consider the different possible values of $x$ separately. 
    Since $a_1>b_2$, we know that $\max(a_1, e_2, b_2) \neq b_2$. 
    Suppose that $\max(a_1, e_2, b_2) = a_1$ so that $x=e_2(a_1-b_2)$. Then \eqref{ineq:rho2-rho1} becomes
    \begin{equation*}
        b_2(a_3-e_2-e_3)-a_1(a_3-e_1-e_2-e_3)>b_2^2+b_2b_3,
    \end{equation*}
    which is equivalent to     \begin{equation}\label{ineq:lhs>0}
        b_2(a_1-e_2-e_3)>a_1(a_3-e_1-e_2-e_3).
    \end{equation}
    Since the right-hand side of \eqref{ineq:lhs>0} is nonnegative, by \eqref{wideness:4}, it follows that $a_1>e_2+e_3$. So, we can use \eqref{wideness:5} to obtain
    \begin{equation}
        b_2(a_1-e_2-e_3)>a_1a_3-a_3e_3-a_2e_2-a_1(a_1-e_2-e_3).
    \end{equation}
    This is equivalent to the inequality $a_2(a_1-e_3)>a_3(a_1-e_3)$, which is false, since $a_2 < a_3$ and $a_1>e_2+e_3>e_3$.

    Finally, consider when $\max(a_1, e_2, b_2) = e_2$ so that $x=a_1(e_2-b_2)$. Then \eqref{ineq:rho2-rho1} becomes
    \begin{equation*}
        b_2(a_3-e_2-e_3)-a_1(a_3-e_1-e_2-e_3)>b_2^2+e_2(a_1-b_2)-a_1(e_2-b_2)+b_2b_3,      
    \end{equation*}
    which is equivalent to 
    \begin{equation*}
        -b_2e_3-a_1(a_3-e_1-e_2-e_3)>0,    
    \end{equation*}
    which contradicts \eqref{wideness:4}.

\bigskip
\noindent{\bf Subcase (ii):\ } $a_1>b_2$ and $s_1=\min(\I_{u+y})$ and $s_2<\max(\I_{y+w})$.

    If $s<\max(\I_{u+v+y+w})$, then we can decrease $u$ and increase $y$ (increasing $s_2$ and $s$). This would yield a contradiction, since we assume that $u$ cannot be decreased. Therefore, $s=\max(\I_{u+v+y+w})$. We will show that this leads to a contradiction.  
    
    By Claim~\ref{claim:s_1=min}, $u+y=s_1=a_1(e_3-b_3)$ and $x_{331}=0$. 
Since $s=s_1+s_2<\max(\I_{u+y})+\max(\I_{y+w})=\max(J_1)$ we must have $u+y+y+w=s=\max(\I_{u+v+y+w})=\max(J_2)=b_3^2+(a_2-b_3)e_3$. 
So, by \eqref{xijk_values}, $x_{333}=0$ and thus $\rho_3=x_{332}=b_2b_3-z_{332}\le b_2b_3$. By \eqref{eq:rho_3}, 
\begin{equation}\label{eq:a3-e3}
        a_3-e_3\le b_2.
    \end{equation}
    
    By \eqref{xijk_values}, $w=z_{322}=x_{222}-x_{322}$. So $x_{222}=x_{322}$, by Claim~\ref{claim:u=maxIu}. Also, since $x_{321}=\max(I_y)-y\geq0$ and $x_{323}=x_{332}=\rho_3>0$ by \eref{eq:rho_3} and \eref{wideness:4}, we have $x_{222}=x_{322}< \rho_2$, by  \eqref{e:rhoj}. Using \eqref{xijk_values}, \eqref{eq:rho_2}, and \eqref{eq:a3-e3}, we have
    \[b_2(b_2-e_2)+a_1e_2-x< b_2(a_3-e_2-e_3)<b_2^2-b_2e_2,\]
    which yields $x>a_1e_2$, contradicting \eqref{ineq:max_min}.    

\bigskip
\noindent{\bf Subcase (iii):\ } $a_1>b_2$ and $s_1>\min(\I_{u+y})$ and $s_2=\max(\I_{y+w})$.

    If $s>\min(\I_{u+v+y+w})$, then we can decrease $u$ (decreasing $s_1$ and $s$). This would yield a contradiction, since we assume that $u$ cannot be decreased. Therefore, $s=\min(\I_{u+v+y+w})$. We will show that this leads to a contradiction.
    
    Now $\min(\I_{u+v+y+w})=s=s_1+s_2 >\min(\I_{u+y})+\min(\I_{y+w})=\min(J_1)$. 
    So, $s=\min(J_2)=(a_2-b_3)e_3$. Thus, $z_{333}=0$ by \eqref{def:z333}, and $x_{333}=x_{233}=b_3^2$ by \eqref{xijk_values}.
    By Claim~\ref{claim:s_2=max}, $y+w=s_2=b_2e_3$. So $z_{332}=0$, by \eqref{def:z333}. Hence, $x_{323}=x_{332}=x_{232}=b_2b_3$, by \eqref{xijk_values}. From \eqref{e:rhoj}, \begin{equation}\label{eq:x331_eq}
        x_{331}=\rho_3-x_{332}-x_{333}=\rho_3-x_{232}-x_{233}.
    \end{equation}
     Since $x_{222}-x_{322}=z_{322}=w=0$, it follows that $x_{322}=x_{222}$. Also, $x_{221}=a_1b_2-a_1e_2+x=\max(I_y)$. But $x_{321}=x_{221}-z_{321}=\max(I_y)-y>0$ by Claim~\ref{claim:y<max}.  By \eqref{e:rhoj} we have
     \begin{equation}\label{eq:rho2_bigger}
         \rho_2=x_{321}+x_{322}+x_{323}>x_{322}+x_{323}=x_{222}+x_{232},
     \end{equation}
     and 
     \begin{equation}\label{eq:x321_eq}
         x_{321}=\rho_2-x_{222}-x_{232}.
     \end{equation}
     Note that $x_{211}=a_1(a_1-e_1)-z_{211}=a_1(a_1-e_1)-x=\max(I_u)$ so $x_{311}=x_{211}-z_{311}=\max(I_u)-u<0$. By \eqref{e:rhoj}, \eqref{eq:x331_eq} and \eqref{eq:x321_eq}, $$\rho_1=x_{311}+x_{312}+x_{313}<x_{312}+x_{313}=x_{321}+x_{331}=\rho_2-x_{222}-x_{232}+\rho_3-x_{232}-x_{233}.$$ This is equivalent to
     \begin{equation}\label{eq:all_rhos}
 2x_{232}+x_{233}+x_{222}<\rho_2+\rho_3-\rho_1.
     \end{equation}
     Since $a_1>b_2$, it follows from \eqref{eq:value_of_x} that $x=e_2(a_1-b_2)$ or $x=a_1(e_2-b_2)$ (depending on whether $a_1\ge e_2$). If $x=a_1(e_2-b_2)$, then $x_{212}=a_1b_2-a_1e_2+x=0$, contradicting that  $x_{212}\ge x_{312}=x_{321}>0$. Hence, 
     \begin{equation}\label{eq:x_equals:e2_times_a1-b2}
         x=e_2(a_1-b_2).
     \end{equation}
    The terms on the left-hand side of \eqref{eq:all_rhos} are $x_{232}=b_2b_3$, $x_{233}=b_3^2$, and $x_{222}=b_2^2-(b_2e_2-a_1e_2+x)=b_2^2$, by \eqref{eq:x_equals:e2_times_a1-b2}.
    Substituting these values and \eref{eq:rho_2} into \eqref{eq:rho2_bigger} yields $b_2(a_3-e_2-e_3)>b_2^2+b_2b_3$, which is equivalent to
    \begin{equation}\label{ineq:a1_morethan_e2+e3}
        a_1>e_2+e_3.
    \end{equation}
    Substituting the values for $x_{232}$, $x_{233}$, and $x_{222}$, as well as the expressions for $\rho_1$, $\rho_2$, and $\rho_3$ from \eqref{eq:rho_1}, \eqref{eq:rho_2} and \eqref{eq:rho_3} into \eqref{eq:all_rhos} yields
    \begin{equation*}
        (b_2+b_3)^2<(b_2+b_3)(a_3-e_3)-b_2e_2-a_1a_3+a_1(e_1+e_2+e_3).
    \end{equation*}
    By \eqref{wideness:5} and \eqref{ineq:a1_morethan_e2+e3}, 
    \begin{equation*}
        (b_2+b_3)^2<(b_2+b_3)(a_3-e_3)-b_2e_2-a_1a_3+a_3e_3+a_2e_2+a_1(a_1-e_2-e_3) = (b_2+b_3)^2,
    \end{equation*}
    a contradiction, which completes the proof of Subcase (iii) and hence also of Case (ii).
\end{proof}

\tref{thm:wide_implies_alloc} now follows.
\end{proof}

\begin{remark}
    Not every allocation of the first two row blocks of a wide Young diagram with three row lengths can be extended to an allocation of the third row block. 
    For example, let $Y$ be the wide Young diagram $(5,4,3,3)$ (where $a_1=3,a_2=4$, $a_3=5$,  $e_1=2$, and $e_2=e_3=1$). Let $Y'$ be the subdiagram of $Y$ consisting of the three lower rows (two lower row blocks). Consider the following allocation $(z_{ijk})$ of $Y'$: $z_{111} = 6$, $x=z_{211} = 3$, $z_{212} = 0$, $z_{221} = 0$, and $z_{222} = 1$. This allocation is demonstrated visually in Figure~\ref{fig:bad_allocation}. In this figure, the number of occurrences of symbol $k$ in the intersection of row block $i$ and column block $j$ is the value $z_{ijk}$.    
    
    \begin{figure}[h!]
    \centering
    \includegraphics[width=3.0cm]{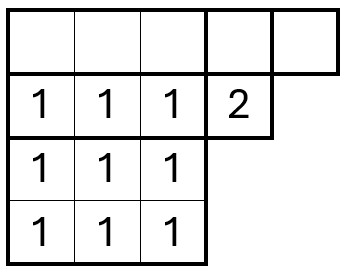}
    \caption{An allocation for $(4,3,3)$ that cannot be extended to an allocation for $(5,4,3,3)$.}
    \label{fig:bad_allocation}
    \end{figure}
    
    Lamentably, not all the constraints in \eqref{ineq:uvyw_nonneg}--\eqref{ineq:all_less} can hold. For allocating the upper row block, we have $z_{311}=u\le a_1(a_1-e_1)-x=3\cdot(3-2)-3=0$, by \eqref{ineq:u_less}. So, $u$ must be zero. 
    Also, $z_{322}=w\le a_1e_2-b_2(e_2-b_2)-x=0$. So, $w=0$. 
    From \eqref{ineq:ub_sums_of_2} we have $v+w\le 1$ and $y+w\le 1$. So, $v\le 1$ and $y\le 1$. 
    On the other hand, from \eqref{ineq:ub_sum_of_4} we must have $u+v+y+w\ge 3$. So, $v+y\ge 3$, a contradiction.
    Thus, the given allocation $(z_{ijk})$ cannot be extended to an allocation for $Y$. 
    However, if we choose values $(z'_{ijk})$ using \eqref{eq:value_of_x}, we have $z'_{111}=6$, $z'_{211}=2$, $z'_{212}=1$, $z'_{221}=1$, and $z'_{222}=0$, and this can be extended to an allocation of $Y$ by taking, 
    $z'_{311}=1$, $z'_{312}=1$, $z'_{313}=1$, $z'_{321}=1$, $z'_{331}=1$, with all other values $z'_{3jk}$ being $0$.
\end{remark}

\section{Discussion: towards the general case}\label{sec:general}

It is natural to try to generalise the construction in Section 4 to Young diagrams with any number of row lengths. 
Let $Y$ be a wide Young diagram with $p$ row blocks.
We assume the subdiagram $Y'$ consisting of the first $p-1$ blocks has an allocation $(z'_{ijk})_{1\le k,j\le i\le p-1}$ adhering to certain, yet unknown, assumptions $(X)$ generalizing \eqref{eq:value_of_x}, and with the additional symmetry property $z'_{ijk}=z'_{ikj}$ for all $1\le k,j\le i\le p-1$.
We aim to find an allocation  $(z_{ijk})_{1\le k,j\le i\le p}$ of $Y$ that is an extension of $(z'_{ijk})$ (i.e.\ $z_{ijk}=z'_{ijk}$ for $1\le k,j\le i\le p-1$) and that satisfies the symmetry property $z_{pjk}=z_{pkj}$ for all $1\le k,j\le p$. 

If we choose values $z_{pjk}$ for $1\le j\le k\le p-1$ satisfying ($X$), then the rest of the allocation will be determined by \eqref{e:sumk} and \eqref{e:sumj}:
\begin{equation}\label{def:z_pjp}
    z_{pjp}=b_je_p-\sum_{k=1}^{p-1}z_{pjk}\quad j=1,\ldots,p-1,
\end{equation}
\begin{equation}\label{def:z_ppk}
    z_{ppk}=b_ke_p-\sum_{j=1}^{p-1}z_{pjk}\quad k=1,\ldots,p-1,
\end{equation}
and
\begin{equation}\label{def:z_ppp}
    \begin{split}
        z_{ppp}&=b_pe_p-\sum_{k=1}^{p-1}z_{ppk}
        =b_pe_p-\sum_{k=1}^{p-1}\left(b_ke_p-\sum_{\smash{j}=1}^{p-1}z_{pjk}\right)\\
        &=b_pe_p-e_p\sum_{k=1}^{p-1}b_k+\sum_{k=1}^{p-1}\sum_{j=1}^{p-1}z_{pjk}
        =(b_p-a_{p-1})e_p+\sum_{k=1}^{p-1}\sum_{j=1}^{p-1}z_{pjk}.
    \end{split}
\end{equation}

If we impose the constraints on the sums in \eqref{def:z_pjp}, \eqref{def:z_ppk}, and  \eqref{def:z_ppp} arising from the non-negativity constraints of each variable and from \eqref{e:sumi}, then we obtain upper and lower bounds for each of the variables $z_{pjk}$ and each of the sums $\sum_{k=1}^{p-1}z_{pjk}$, $\sum_{j=1}^{p-1}z_{pjk}$, and $\sum_{j=1}^{p-1}\sum_{k=1}^{p-1}z_{pjk}$.
We can show that wideness implies that each lower bound is indeed smaller than or equal to the corresponding upper bound. However, so far we have not been able to show that the system of inequalities is consistent, that is, that we can actually choose non-negative values $z_{pjk}$ such that all our constraints are satisfied.
The difficulty of showing this for the case $p=3$ in Section~\ref{sec:three} suggests that new ideas should be considered.

Another possible approach is considered in \cite{abgk}, where the possible entries of a Young diagram are viewed as edges in a 3-partite 3-hypergraph $H(Y)$. Consider the dual notions of the second matching number $\nu^{(2)}(H(Y))$ and the second covering number $\tau^{(2)}(H(Y))$ (for definitions of $H(Y), \nu^{(2)}$, and $\tau^{(2)}$ see subsection 1.1 in \cite{abgk}). From the definition of $\nu^{(2)}(H(Y))$, we have that a Young diagram is Latin if and only if $\nu^{(2)}(H(Y))=|Y|$. The latter implies that Conjecture~\ref{conj:wpc} is equivalent to the claim that if a Young diagram $Y$ is wide, then $\nu^{(2)}(H(Y))=|Y|$. In \cite{abgk}, the weaker statement that if $Y$ is wide then $\tau^{(2)}(H(Y))=|Y|$ is proved (weaker, since $\tau^{(2)}(H)\ge\nu^{(2)}(H)$ for any hypergraph $H$). So, in view of Theorem~\ref{thm:alloc_implies_latin}, the following is equivalent to Conjectures~\ref{conj:wpc} and \ref{conj:wide_implies_outline}:

\begin{conjecture}
    For a Young diagram $Y,$ if $\tau^{(2)}(H(Y))=|Y|$, then $Y$ has an allocation.
\end{conjecture}


\bibliographystyle{siam}
\bibliography{refs}

@article{chow2003,
  title={Wide partitions, Latin tableaux, and {R}ota's basis conjecture},
  author={Chow, Timothy Y and Fan, C Kenneth and Goemans, Michel X and Vondrak, Jan},
  journal={Adv. Appl. Math.},
  volume={31},
  number={2},
  pages={334--358},
  year={2003},
  publisher={Elsevier}
}

@article{rota94,
  title={On the relations of various conjectures on Latin squares and straightening coefficients},
  author={Huang, Rosa and Rota, Gian-Carlo},
  journal={Discrete Math.},
  volume={128},
  number={1-3},
  pages={225--236},
  year={1994},
  publisher={North-Holland}
}

@article{chow_tief,
  title={The {L}atin Tableau Conjecture},
  author={Chow, Timothy Y and Tiefenbruck, Mark G},
journal={Electron. J. Combin.},
volume={32(2)},
pages={P2.48},
  year={2025}
}

@article{abgk,
  title={2-covers of wide {Y}oung diagrams},
  author={Aharoni, Ron and Berger, Eli and Guo, He and Kotlar, Dani},
  journal={arXiv preprint arXiv:2311.17670},
  year={2023}
}

@article {Hilton,
    AUTHOR = {Hilton, A. J. W.},
     TITLE = {The reconstruction of {L}atin squares with applications to
              school timetabling and to experimental design},
   JOURNAL = {Math. Programming Stud.},
  FJOURNAL = {Mathematical Programming Study},
    VOLUME = {13},
      YEAR = {1980},
     PAGES = {68--77},
      ISSN = {0303-3929},
   MRCLASS = {90B35 (05B15 62K10)},
  MRNUMBER = {592087},
       DOI = {10.1007/bfb0120908},
       URL = {https://doi.org/10.1007/bfb0120908},
}

\end{document}